\documentclass{amsart}

\usepackage{mathtools, amssymb, hyperref}
\newtheorem{theorem}{Theorem}[section]
\newtheorem{lemma}[theorem]{Lemma}

\theoremstyle{definition}
\newtheorem{definition}[theorem]{Definition}
\newtheorem{example}[theorem]{Example}

\newtheorem{proposition}[theorem]{Proposition}
\newtheorem{corollary}[theorem]{Corollary}

\theoremstyle{remark}
\newtheorem{remark}[theorem]{Remark}

\numberwithin{equation}{section}

\begin{document}

\title[Parseval identities for Gowers norms]{Parseval-type identities for  Gowers \\ uniformity norms in finite abelian groups}

\author{Martin Niepel}
\address{Faculty of Mathematics, Physics, and Informatics, Comenius University, \newline 842 48 Bratislava, Slovakia}
\email{mniepel@fmph.uniba.sk}


\date{July 14, 2025} 


\begin{abstract}
Orthogonality relations for cubes of characters in Gowers inner products $\langle \cdot \rangle_{d,l}$ lead to
Parseval-type identities and isometries for suitably generalized Gowers uniformity norms $U^{d,l}$.  
\end{abstract}


\maketitle

\section{Introduction}
\label{sec:intro}

In \cite{Gow01}, Gowers introduced uniformity norms $U^d(G)$ for complex-valued functions on the group
$G=\mathbb{Z}/N\mathbb{Z}$ as part of his new proof of Szemer\'edi's theorem. This concept has led to numerous breakthroughs in additive combinatorics. A similar and independent approach by Host and Kra, \cite{HoKr04}, \cite{HoKr05}, \cite{HoKr08}, 
has also been influential in ergodic theory. The statement and subsequent proof of the inverse theorem for Gowers norms appeared in \cite{GrTa08} and \cite{GrTaZi12}. 

We would also like to direct the reader to an alternative treatment of higher-order Fourier theory through the framework of nilspaces, as developed by Camarena and Szegedy; see \cite{CamSze10}, \cite{Szeg12}, \cite{Can16}, and further references therein. 

In the present paper, we describe variants of uniformity norms $U^{d,l}$ of order $d$ and degree $l$ in the space $\mathbb{C}^G$ of complex-valued functions on a finite abelian group $G$ (or on  $\mathbb{C}^{\widehat{G}}$) and the related $2^d$-ary inner products $\langle \cdot \rangle_{d, l}$. We prove that these norms satisfy Parseval-type identities when evaluated on a cube of functions $F: G^{\{0,1\}^d} \to \mathbb{C}$ and its Fourier transform $\widehat{F}: \widehat{G}^{\{0,1\}^d} \to \mathbb{C}$, with $F=\left( f_\omega\right)_{\omega \in \{0,1\}^d}$, $\widehat{F}=\left( \widehat{f}_\omega\right)_{\omega \in \{0,1\}^d}$, and $\widehat{(f_\omega)} = \widehat{f}_\omega$ for all $\omega \in \{0,1\}^d$:
\begin{equation}\label{ParsId}
\left\langle F \right\rangle_{d,l} =  \left\langle \widehat{F}
\right\rangle_{d,d-l-1} .
\end{equation}

The classical Parseval's theorem for finite abelian groups, which states that the Fourier transform is unitary, 
turns out to be a special case with the choice of parameters $d=1$ and $l=0$. Similarly, the known relation between the $U^2(G)$ norm of a function $f$ and the $\ell_4(\widehat{G})$ norm of its Fourier transform $\widehat{f}$ corresponds to the case $d=2$ and $l=1$. We obtain the first new and interesting result for $d=3$ and $l=1$, which served as the starting point for this line of inquiry. In this case, the $U^3(G)$ norm of $f$ and the $U^3(\widehat{G})$ norm of $\widehat{f}$ are equal:
\begin{equation} \|f\|_{U^3(G)} = \|\widehat{f}\|_{U^3(\widehat{G})}.
\end{equation}

Later, we will show that relation \eqref{ParsId} is a special instance of a more general higher-order Poisson formula \eqref{Poisson}. We will also establish some elementary inequalities between higher-degree norms and present some equalities for corner convolutions and their Fourier transforms. 

The paper is organized as follows. In Section \ref{sec:HDC}, we introduce the notation and define higher degree cubes $\mathcal{P}^{d,l}(G)$ in $G^{\{0,1\}^d}$. We explore their structure in detail and deduce factorizations \eqref{Pdecomp} and \eqref{Pdecomp2}. In Section \ref{sec:norms}, we generalize the concept of multiplicative derivative to higher degrees and show how it splits in Proposition \ref{prop:dersplit}. As a consequence, we establish the desired properties of $2^d$-ary inner products $\langle \cdot \rangle_{d, l}$.

In Section \ref{sec:orto-fin}, we explore the higher order orthogonality of characters and consequently prove Identity \eqref{ParsId}. In Section \ref{sec:Poisson}, we outline an alternative approach to higher order Parseval's identity using higher order Poisson summation formula \eqref{Poisson} and apply it to corner convolutions \eqref{Corner}.
 
{\bf Acknowledgement (Funding):} 
The author was supported by the project VEGA 1/0596/21 from the Ministry of Education of Slovakia. 


\section{Higher Degree Cubes}
\label{sec:HDC}

There are various notation conventions used in the literature for dealing with cubes, Host-Kra groups, and nilspaces. We will primarily follow the conventions outlined in \cite{Tao12}, specifically Chapters 1.3 and 2.2, but will adjust slightly when necessary.

Let $G$ be a finite abelian group and $\widehat{G}$ its Pontryagin dual group of additive characters. The space of complex-valued functions on $G$ will be denoted by $\mathbb{C}^G$. Equipped with the standard Hermitian product
$$ \langle f, g \rangle = \mathbb{E}_{x \in G} f(x)\overline{ g(x)},$$ 
it forms a finite-dimensional Hilbert space $L^2(G)$, where the characters in $\widehat{G}$ provide an orthonormal basis. While we assume throughout the paper that $G$ is finite, most definitions and results also extend to compact abelian Lie groups (e.g., $\mathbb{R}/\mathbb{Z}$).

A finite set $A$ with $d$ elements, typically $A=\{1, 2, \dots, d\}$, will be used as a label set. Given a linear order on $A$, elements in a $d$-dimensional binary cube $\{0,1\}^A$ can be written in a standard manner, for example, $ \omega = \omega_1 \omega_2  \dots \omega_d$ for $A=\{1,2,\dots,d\}$.  The \textit{weight} of $\omega \in \{0,1\}^A$, denoted by $|\omega|$, is the sum $\sum_{i\in A} \omega_i$.  

For every $i \in A$, there is a natural identification between $\{0,1\}^A$ and $\{0,1\}^{A \setminus \{ i\} } \times \{0,1\}^{\{i\} }$, where $\omega$ corresponds to a pair of restrictions $(\omega_{|A\setminus \{i\}}, \omega_{|\{i\}}) $. Alternatively, we will use the slightly shorter notation $\omega=(\omega', \omega_i)$. The projection onto the $i$-th coordinate, denoted by $\pi_i : \{0,1\}^ A \to \{0,1\}^{\{i\}}$, sends $\omega$ to $\omega_i$, while the complementary projection $\bar{\pi}_i : \{0,1\}^A \to \{0,1\}^{A \setminus \{i\}}$ forgets the $i$-th coordinate and satisfies $\bar \pi_i (\omega) = \omega'$. Additionally, we have two inclusions $s_i^0$ and $s_i^1$ of a $(d-1)$-dimensional cube $\{0,1\}^{A \setminus \{i\}}$ into $\{0,1\}^A$ with constant $i$-th coordinate. Specifically, $s_i^a : \{0,1\}^{A \setminus \{i\}} \to \{0,1\}^{A} $ sends $\omega'$ to $(\omega',a_i)$, where $a_i \in \{0,1\}$. The image of $s_i^a$ is the preimage $\pi_i^{-1}(a_i)$. 
This is equivalent to requiring $\bar \pi_i \circ s_i^a = \mathrm{id}_{\{0,1\}^{A \setminus \{i\}}}$ and $\pi_i \circ s_i^a \equiv a$ for $a \in \{0,1\}^{\{i\}}$.

Similarly, for any $B\subseteq A$ with $|B|=k$, we can write $\omega = (\omega_{|A\setminus B}, \omega_{|B}) \in  \{0,1\}^{A \setminus B} \times \{0,1\}^{B}$. We have projections $\pi_B : \{0,1\}^ A \to \{0,1\}^{B}$ and $\bar{\pi}_B : \{0,1\}^A \to \{0,1\}^{A \setminus B}$, as well as inclusions $s_{B}^{\omega_{|B}}: \{0,1\}^{A\setminus B} \to \{0,1\}^A$, where $\omega_{|B} \in \{0,1\}^B$. The inclusion $s_{B}^{\omega_{|B}}$ maps the $(d-k)$-dimensional cube $\{0,1\}^{A \setminus B}$ to the $(d-k)$-dimensional face $\pi_B^{-1}(\omega_{|B})$ of $\{0,1\}^{A}$, where the $B$-coordinates are constant and equal to $\omega_{|B} \in \{0,1\}^B$. These conditions are equivalent to requiring $\bar \pi_B \circ s_{B}^{\omega_{|B}} = \mathrm{id}_{\{0,1\}^{A \setminus B}}$ and $\pi_B \circ s_B^{\omega_{|B}} \equiv \omega_{|B}$.

We define an \textit{$A$-labeled $d$-dimensional cube in $G$}, or an \textit{$A$-cube in $G$}, as a point $\left( x_\omega \right)_{\omega \in \{0,1\}^A} \in G^{\{0,1\}^A}$. Alternatively, it can be viewed as the image $p(\{0,1\}^A)$ of the abstract binary cube $\{0,1\}^A$ under a map $p: \{0,1\}^A \to G$. In the rest of the paper, we will use both notations --- a point $(x_\omega)$ and a map $p$ --- interchangeably.

The group operation from $G$ can be lifted pointwise to $G^{\{0,1\}^A}$. Consequently, the set of $A$-cubes in $G$ forms an abelian group isomorphic to $G^{2^d}$. 

For every $i \in A$, we can split the $2^d$-tuple $p \in G^{\{0,1\}^A}$ into a pair of two $2^{d-1}$-tuples $(p_{0_i},p_{1_i})= (p \circ s^0_i,p\circ s^1_i) \in G^{\{0,1\}^{A \setminus \{i\} }} \times G^{\{0,1\}^{A \setminus \{i\} }}$. Next, we 
define a boundary operator $\delta_i : G^{\{0,1\}^A} \to G^{\{0,1\}^{A \setminus \{i\} }}$ from the space of $A$-cubes in $G$ to the space of $A \backslash \{i\}$-cubes in $G$ by
\begin{equation}
    \delta_i p  \coloneqq p \circ s_i^0 - p \circ s_i^1.
\end{equation}
This boundary operator assigns differences in the $i$-th direction to the $2^{d-1}$ edges (1-dimensional faces) of the cube $\{0,1\}^{A}$ along that direction. Specifically, for $\omega' \in \{0,1\}^{A \setminus \{i\}}$ we have
\begin{equation*}
    (\delta_i p )(\omega') = p(s_i^0(\omega')) - p (s_i^1 (\omega')) = x_{(\omega',0)} - x_{(\omega',1)}.
\end{equation*}
For the boundary operators $\delta_i : G^{\{0,1\}^A} \to G^{\{0,1\}^{A \setminus \{i\} }}$, $\delta_j : G^{\{0,1\}^A} \to G^{\{0,1\}^{A \setminus \{j\} }}$, $\delta_i' : G^{\{0,1\}^{A\setminus \{j\}}} \to G^{\{0,1\}^{A \setminus \{i,j\} }}$ and $\delta_j' : G^{\{0,1\}^{A\setminus \{i\}}} \to G^{\{0,1\}^{A \setminus \{i,j\} }}$, it is straightforward to check that the compositions $\delta_j' \circ \delta_i$ and $\delta_i'\circ \delta_j$ are equal.  

Hence, for any $B=\{i_1, i_2, \dots , i_k\} \subseteq A$, we can define a boundary operator $\delta_B : G^{\{0,1\}^A} \to G^{\{0,1\}^{A \setminus B}}$ as 
\begin{equation*}
\delta_B \coloneqq \delta_{i_k}^{(k-1)} \circ  \dots  \circ \delta_{i_2}'\circ \delta_{i_1}
\end{equation*}
This definition is independent of the order in which the boundary operators are composed. 

Another way to interpret the maps $\delta_B$ is that they represent alternating sums along Gray codes over the $k$-dimensional faces $\pi_{A \setminus B}^{-1}(\omega_{|A\setminus B})=s_{A \setminus B}^{\omega_{|A\setminus B}} \left(\{0,1\}^{B}\right)$ of the cube $\{0,1\}^A$. Specifically, for any $p \in G^{\{0,1\}^A}$ and $\omega_{|A\setminus B} \in \{0,1\}^{A \setminus B}$, we have
\begin{equation}
\label{gray}
    (\delta_B p)(\omega_{|A\setminus B}) = (-1)^{\left|\omega_{|A\setminus B}\right|} \sum_{\omega_{|B} \in \{0,1\}^{B}\ } (-1)^{\left|\omega_{|A\setminus B}\right|+ \left|\omega_{|B}\right|}x_{(\omega_{|A\setminus B}, \omega_{|B})}.
\end{equation}

Note also that the map $\delta_B$ is a group homomorphism from $G^{\{0,1\}^A}$ to $G^{\{0,1\}^{A\setminus B}}$. Consequently, its kernel is an abelian subgroup of $G^{\{0,1\}^A}$. This leads to the following definition.

\begin{definition}
\label{def:HDcubes} For $l = -1, \dots, d-1$ define the space of \textit{$A$-labeled $d$-dimensional degree-$l$ cubes in $G$}  as  
$$ \mathcal{P}^{A,l}(G) \coloneqq \bigcap_{\substack{B \subseteq A \\  |B|=l+1}} \ker \delta_B.$$
For $l \ge|A|$, set
$$ \mathcal{P}^{A,l}(G) \coloneqq G^{\{0,1\}^A}.$$
\end{definition}

\begin{example}
Using this notation, any (non-zero) constant map $p=(x)_{\omega \in \{0,1\}^A}$ has degree $0$, and the unique cube of degree $l=-1$ is the zero point $(0)_{\omega \in \{0,1\}^A}$. 

Constant maps from $\{0,1\}$ to $G$ are related to the standard Hermitian product in $\mathbb{C}^G$, as they represent the diagonal $D=\{(x,x)\in G^2\}=\mathcal{P}^{1,0}(G)$. Similarly, higher-dimensional diagonals $\mathcal{P}^{d,0}$ in $G^{2^d}$ play a role in the definitions of standard $L^{2^d}(G)$ norms.

Affine-linear maps from $\{0,1\}^d$ to $G$ -- parallelepipeds of dimension $d$, which are essential in the definition of the Gowers norm $U^d$ -- form the space $\mathcal{P}^{d,1}(G)$ of $d$-cubes in $G$ of degree $1$. 
\end{example}

From the construction, we see that a higher degree cube  $ \mathcal{P}^{A,l}(G)$ is symmetric with respect to the symmetry group $S_A \ltimes (\mathbb{Z}/2\mathbb{Z})^A$ of the cube $\{0,1\}^A$. We will not introduce any special notation for the action of this symmetry group and will instead use the symmetries implicitly in later computations.

It follows directly from the definition that the degree induces a filtration on the space of all $A$-cubes in $G$:
\begin{equation}
\label{Pfilter}    
\{ (0)_{\omega \in \{0,1\}^A}\}=\mathcal{P}^{A,-1}( G) \subset \mathcal{P}^{A,0}( G) \subset \cdots \subset \mathcal{P}^{A,d}(G) = G^{\{0,1\}^A}.\end{equation}
Moreover, for any $B \subseteq A$ we have two types of  maps  from  $G^{\{0,1\}^A}$ to  $G^{\{0,1\}^{A \setminus B}}$. The first map, $C_{s_{B}^{\omega_{|B}}}$, is a composition operator given by the inclusion $ s_{B}^{\omega_{|B}}$ of $G^{\{0,1\}^{A \setminus B}}$ into $G^{\{0,1\}^A}$. This operator restricts a $d$-dimensional $A$-cube in $G$ to its $(d-k)$-dimensional face $\pi_{B}^{-1}(\omega_{| B}) = s_{B}^{\omega_{|B}} \left( \{0,1\}^{A\setminus B}\right)$ in the $A\setminus B$-direction. The second map, $\delta_B$, assigns alternating Gray sums along faces in the $B$-direction. Both preserve the filtration: the former keeps the degrees unchanged, while the latter shifts them down by $|B|=k$. 

\begin{proposition}
\label{prop:dBshift}
For a subset $B$ of the label set $A$ and $\omega_{|B} \in \{0,1\}^B$, we have 
\renewcommand{\labelenumi}{\roman{enumi})}
\begin{enumerate}
    
\item $ C_{s_{B}^{\omega_{|B}}} \left(\mathcal{P}^{A,l}(G)\right) = \mathcal{P}^{A\setminus B,l}(G)$,

\item $ \delta_B \left(\mathcal{P}^{A,l}(G)\right) = \mathcal{P}^{A\setminus B,l -|B|}(G)$.
\end{enumerate}
\end{proposition}
\begin{proof} 
i) Let $C \subseteq A\setminus B$ with $|C|=l+1$. Since $C$ is a subset of both $A$ and $A\setminus B$, we need to distinguish between the map $\delta_C$  from $G^{\{0,1\}^A}$ to $G^{\{0,1\}^{A \setminus C}}$ and the map $\delta_C'$ from $G^{\{0,1\}^{A\setminus B}}$ to $G^{\{0,1\}^{A \setminus (B \cup C)}}$. 
For $p \in \mathcal{P}^{A,l}(G)$, we have $p \in \ker \delta_C$ by definition.  The operator $C_{s_{B}^{\omega_{|B}}}$ acts on $p \in G^{\{0,1\}^A}$ by composition:
$$C_{s_{B}^{\omega_{|B}}}(p) = p \circ s_{B}^{\omega_{|B}}.$$
Thus, $ \delta_C' (p \circ s_{B}^{\omega_{|B}}) = (\delta_C p )\circ {s'}_{B}^{\omega_{|B}} $  is a zero map from $G^{\{0,1\}^{A \setminus (B \cup C)}}$
to $G$, where ${s'}_{B}^{\omega_{|B}}$ denotes the inclusion of an ${A \setminus (B \cup C)}$-cube into an ${A \setminus C}$-cube with constant $B$-coordinates equal to $\omega_{|B}$.

Therefore, we have an inclusion $C_{s_{B}^{\omega_{|B}}} \left(\mathcal{P}^{A,l}(G)\right) \subseteq \mathcal{P}^{A\setminus B,l}(G)$.

For $p' \in \mathcal{P}^{A\setminus B,l}(G)$ consider the pullback $p=\bar \pi_B^*(p')$, that is, the $B$-cube $(p')_{\omega_{|B} \in \{0,1\}^B}$ that consists of $2^{|B|}$ copies of $p'$. Then, clearly, we obtain $p'$ as a restriction $p \circ s_{B}^{\omega_{|B}}$ of $p$. Since $\delta_i p = 0$ for any $i \in B$, it suffices to check $\delta_C$ for $C \subseteq A \setminus B$ with $|C|=l+1$. Then we have:
$\delta_C(p)=(\delta'_C p')_{\omega_{|B} \in \{0,1\}^B} = (0)_{\omega_{|B} \in \{0,1\}^B}$. This means that $p \in \mathcal{P}^{A,l}(G)$, and we have established the identity: $ C_{s_{B}^{\omega_{|B}}} \left(\mathcal{P}^{A,l}(G)\right) = \mathcal{P}^{A\setminus B,l}(G)$.

ii) By definition, on the right-hand side we have:
$$ \mathcal{P}^{A\setminus B,l -|B|}(G) =  \bigcap_{\substack{C \subseteq A \setminus B \\  |C|=l-|B| +1}} \ker \delta_C.$$
For any $C \subseteq A\setminus B$ the sets $B$ and $C$ are disjoint. Thus, $|B \cup C|=l +1$, and we have $\delta_{B \cup C} = \delta_C \circ \delta_B$. Consequently, we have $\delta_B(\ker \delta_C \circ \delta_B) \subseteq \ker \delta_C$.
Therefore 
$$\delta_B \left(\mathcal{P}^{A,l}(G)\right) = \delta_B \left(  \bigcap_{\substack{D \subseteq A \\  |D|=l +1}} \ker \delta_{D} \right) \subseteq  \delta_B \left( \bigcap_{\substack{C \subseteq A \setminus B \\  |C|=l-|B| +1}} \ker \delta_{B \cup C} \right) \subseteq$$
$$\subseteq \bigcap_{\substack{C \subseteq A \setminus B \\  |C|=l-|B| +1}} \delta_B(\ker \delta_C \circ \delta_B) 
\subseteq \bigcap_{\substack{C \subseteq A \setminus B \\  |C|=l-|B| +1}} \ker \delta_C =  \mathcal{P}^{A\setminus B,l -|B|}(G).$$
To verify the reverse inclusion, take any $p' \in  \mathcal{P}^{A\setminus B,l -|B|}(G)$. We extend $p'$  by $2^{|B|}-1$ zero maps to obtain an  $A$-cube $p=(p', 0, \dots, 0)$. Using equation \eqref{gray}, we have $\delta_B p = p'$. We need to check that  $p \in \mathcal{P}^{A,l}(G)$. For every $D\subseteq A$ with $|D|=l+1$, we have $\delta_D= \delta_{D \cap B}\circ\delta_{D\setminus B}$.  Since $|D\setminus B|\ge l - |B| + 1$, $p' \in \ker \delta_{D\setminus B}$. Hence $\delta_{D\setminus B} p' = 0 $. Thus, $p \in \mathcal{P}^{A, l}(G)$ as it satisfies $ \delta_D p = 0$ for all $D \subseteq A$ with $|D| = l + 1$.
This proves that:
\[
\delta_B \left(\mathcal{P}^{A, l}(G)\right) = \mathcal{P}^{A \setminus B, l - |B|}(G).
\]

\end{proof}

There is more to be said about the structure of higher degree cubes in $G$.
Since $\ker \delta_i$ consists of $A$-cubes that are constant in the $i$-th direction, these cubes can be expressed as $(p_{1_i},p_{1_i})$, where $p_{1_i}$ is some $A\backslash \{i\}$-cube in $G$. This identification corresponds to the pullback ${\bar \pi_i}^* : G^{\{0,1\}^{A\setminus\{i\}}} \to G^{\{0,1\}^A}$, and the composition operator $C_{s^1_i}$ maps $(p_{1_i},p_{1_i})$ back to  $p_{1_i}$ . Therefore, by Proposition \ref{prop:dBshift}, for every degree $l$  we obtain a short exact sequence:

$$ 0 \longrightarrow \mathcal{P}^{A\setminus \{i\},l}(G) \xrightarrow{\phantom{i}{\bar \pi_i}^*\phantom{i}}  \mathcal{P}^{A,l}(G) \xrightarrow{\phantom{i}\delta_i\phantom{i}}  \mathcal{P}^{A\setminus \{i\},l-1}(G) \longrightarrow 0.$$

This exact sequence splits, leading to the isomorphism: 

\begin{equation}
\label{Pdecomp}
    \mathcal{P}^{A,l}(G) \cong  \mathcal{P}^{A\setminus \{i\},l}(G) \times \mathcal{P}^{A\setminus \{i\},l-1}(G).
\end{equation}

Hence, every degree-$l$ $A$-cube $p$ in $G$ can be uniquely reconstructed from a pair $[q_1,q_2]$ of two $A\backslash \{i\}$-cubes in $G$, one of degree $l$ and the second of degree $l-1$, using the maps   
$$ p = (p_{0_i},p_{1_i}) \mapsto [p \circ s_i^1 , \delta_i p] = [p_{1_i},p_{0_i}-p_{1_i}]= [q_1,q_2]$$ 
and 
$$ p(\omega', \omega_i) = \begin{cases} q_1(\omega')+ q_2(\omega'), &\quad \omega_i=0, \\q_1(\omega'), &\quad \omega_i=1. \end{cases}$$

At this point, we would like to note that every degree-$l$ $A$-cube in $G$ is uniquely determined by any of its $l$-dimensional corners. By considering the corner centered at $(0 \dots 0) \in \{0,1\}^A$, which consists of vertices $\omega$ with weight less than or equal to $l$, we
can uniquely assign the value $x_{\tilde\omega}$  to a vertex $\tilde\omega$ with $|\tilde\omega|=l+1$ so that it satisfies the Gray code property \eqref{gray} along the $l+1$-dimensional face $\left\{ \omega | \omega \preceq_A \tilde\omega\right\}$ (here, $\preceq_A$ denotes the coordinate-wise partial order on $\{0,1\}^A$ induced by $0_i \preceq_i 1_i$). Inductively, we can continue with vertices of weight $l+2$, and so on.  
Hence, we conclude that $\mathcal{P}^{A,l}(G)$ is isomorphic to $G^k$, where $k = \sum_{i=0}^l \binom{d}{i} $. 

There are two useful consequences of this observation. Firstly, the factor group $\mathcal{P}^{A,l}(G)/\mathcal{P}^{A,l-1}(G)$ can be identified with the space 
\begin{equation}
\label{DiffP(A,l)}
    \mathcal{D}_{l-1}\mathcal{P}^{A,l}(G) \coloneqq \left\{ (x_\omega) \in \mathcal{P}^{A,l}(G) \, | \, x_\omega = 0 \text{ for } |\omega| \le l-1 \right\},
\end{equation}
consisting of those degree-$l$ cubes whose values vanish on the $(l-1)$-dimensional corner centered at $(0 \dots 0)$. Such points are uniquely determined  by  $ \binom{d}{l} $ values $x_\omega$ with $|\omega| = l$. 

This leads to an isomorphism 

\begin{equation}
\label{Pdecomp2}
    \mathcal{P}^{A,l}(G) \cong  \mathcal{P}^{A\setminus \{i\},l-1}(G) \times \mathcal{P}^{A\setminus \{i\},l-1}(G) \times \mathcal{D}_{l-1}\mathcal{P}^{A \setminus \{i\},l}(G),
\end{equation}
given by  $p = (p_{0_i},p_{1_i}) \mapsto [r_{0_i}, r_{1_i}, r_{d_i}]$. Here, $r_{0_i}$ and $r_{1_i}$ are unique degree-$(l-1)$ $A \backslash \{i\} $-cubes in $G$ determined by $(l-1)$-dimensional corners in $(d-1)$-dimensional faces $\pi_i^{-1}(0_i)$ and $\pi_i^{-1}(1_i)$  of $\{0,1\}^A$ centered at $(0_i0 \dots 0)$ and $(1_i0 \dots 0)$, respectively.
Hence,  $r_{0_i}$ and $r_{1_i}$ are the unique solutions in $\mathcal{P}^{A\setminus \{i\},l-1}(G)$ satisfying
$$ p_{a_i} (\omega', a_i) = r_{a_i}(\omega') \qquad \text{ for } \quad \omega' \in \{0,1\}^{A\setminus \{i\}}, \; |\omega'| \le l-1, \; a_i = 0_i,1_i. $$
The remainder $r_{d_i}\in \mathcal{D}_{l-1}\mathcal{P}^{A \setminus \{i\},l}(G)$ is uniquely determined by  
$$ r_{d_i}(\omega') = \begin{cases} 0   &\qquad   |\omega'| \le l-1,\\
p_{0_i} (\omega', 0_i) - r_{0_i}(\omega') &\qquad   |\omega'| = l. 
\end{cases}$$
By construction, the pair $(r_{d_i},r_{d_i})$ is the unique element in $\mathcal{D}_{l-1}\mathcal{P}^{A,l}(G)$ vanishing on all vertices  $(\omega',\omega_i)$ with $|\omega'|\le l-1$, $\omega_i \in \{0,1\}^{\{i\}}$ and matching $r_{d_i}(\omega')$ on all wertices 
 $(\omega', 0_i)$ for $|\omega'|=l$. Moreover, it satisfies  $r_{d_i}(\omega')= p_{1_i}(\omega', 0_i) - r_{1_i}(\omega')$ as well. Hence, 
\begin{equation}
\label{cornersplit}(p_{0_i},p_{1_i}) = (r_{0_i}+r_{d_i},r_{1_i}+r_{d_i}).
\end{equation} 
This decomposition separates the contributions from the two degree-$(l-1)$ cubes $r_{0_i}$, $r_{1_i}$ and the ``corner defect'' $r_{d_i}$, which will be essential in the proof of the Cauchy-Schwarz-Gowers inequality in Proposition \ref{CSG}.

The second consequence is related to the question of choosing a measure on $\mathcal{P}^{A,l}(G)$, especially when the group $G$ in question is an infinite compact group. In that case, we can use the Haar measure  on $\mathcal{P}^{A,l}(G)$ induced by the isomorphism $\mathcal{P}^{A,l}(G) \cong G^k$, for averaging over the space of degree-$l$ $A$-cubes in $G$.

\begin{remark}
    In the same spirit as $\eqref{DiffP(A,l)}$, the inclusion $\mathcal{P}^{A,l'}(G) \subseteq \mathcal{P}^{A,l}(G)$ for $l'\le l$ induces a splitting  
    $$\mathcal{P}^{A,l}(G) =  \mathcal{P}^{A,l'}(G) \times \mathcal{D}_{l'}\mathcal{P}^{A,l}(G),$$
    where 
    \begin{equation}
\label{Diff2P(A,l)}
    \mathcal{D}_{l'}\mathcal{P}^{A,l}(G) \coloneqq \left\{ (x_\omega) \in \mathcal{P}^{A,l}(G) \, | \, x_\omega = 0 \text{ for } |\omega| \le l' \right\}
\end{equation}
  consists of those cubes  in $\mathcal{P}^{A,l}(G)$ vanishing on the $l'$-dimensional corner centered at $(0 \dots 0)$. These points are then uniquely determined  by the values $x_\omega$ for $l'<|\omega| \le l$. 
\end{remark}

\begin{remark} We note that higher degree cubes in abelian groups (and in their duals) were already treated in \cite{Szeg12}, Section 2.7.

These higher degree cubes are, in fact, special instances of Host-Kra groups $HK^d(G,\le l)$, as discussed in \cite{HoKr05} or Section 1.6. of \cite{Tao12}. The definition of boundary maps $\delta_B = \delta_{i_k}^{(k-1)} \circ  \dots  \circ \delta_{i_2}'\circ \delta_{i_1}$  does not depend on the order of taking differences, provided $G$ is a $k$-step nilpotent group with a suitable filtration.

Note also the connection between the role played by $l$-dimensional corners and the unique completion axiom for nilspaces, as discussed in  \cite{HoKr08}, \cite{Szeg12}, \cite{CamSze10}.
\end{remark}

\section{\texorpdfstring{$2^d$-ary Inner Products and $U^{d,l}$ Norms for Finite Abelian Groups}{2\^d-ary Inner Products and U\^{d,l} Norms for Finite Abelian Groups}}
\label{sec:norms}

The concept of a \textit{multiplicative discrete derivative} of a function $f \in \mathbb{C}^G$ at a point $x \in G$ in the direction $h \in G$ is well established in the literature, such as in  \cite{Tao12}, Section 1.3. This derivative is defined as:
$$ \Delta_h f(x) = f(x+h)\overline{f(x)}.$$ 
This expression can be interpreted as a sesquilinear functional on pairs of functions from $\mathbb{C}^G$. In this context, we modify this definition slightly to better fit our notation. Specifically, we aim to make the expression complex linear in the second argument and use $A$-label notation. For $i \in A$, we define:  

\begin{equation}
\label{productder}
    \overline{\Delta}^{\{i\}}_{h_i} (f_{0_i},f_{1_i})(x) \coloneqq f_{0_i}(x)\overline{f_{1_i}(x+h_i)}= \prod_{\omega_i \in \{0,1\}^{\{i\}}} \mathcal{C}^{\omega_i}f_{\omega_i}(x + \omega_i h_i),
\end{equation}
where $\mathcal{C}$ denotes the complex conjugation operator.

On the other hand, the group $G$ acts on itself by translations, and each group element $y\in G$ defines a shift operator $T^y:\mathbb{C}^G \to \mathbb{C}^G$, given by 
$$T^y f(x) \coloneqq f(x + y).$$
Thus, we can express the multiplicative discrete derivative using a shift operator as follows:
\begin{equation}
\label{translateder}    
  \overline{\Delta}^{\{i\}}_{h_i} (f_{0_i},f_{1_i})(x) = \left(f_{0_i}\overline{T^{h_i} f_{1_i}}\right)(x). 
\end{equation}
Both expressions \eqref{productder} and \eqref{translateder} for the first-order multiplicative discrete derivative can be generalized, to some extent, by utilizing the filtration of $A$-cubes on $G$  from \eqref{Pfilter}. This generalization enables the definition of multiplicative discrete derivatives of higher orders and degrees.
 
\begin{definition}
For a $d$-dimensional $A$-cube of complex valued functions $F =(f_\omega) : G^{\{0,1\}^A} \to \mathbb{C}$,  define its {\it multiplicative discrete derivative of order $d$ and degree $l$} at a point 
$p =\left( x_\omega \right) \in \mathcal{P}^{A,l}(G)$ as 
\begin{equation}\label{disc_der}
 \overline{\Delta}^A_{p,l} (F)  \coloneqq \prod_{\omega \in \{0,1\}^A} \mathcal{C}^{|\omega|} f_\omega (x_\omega).
\end{equation}

\end{definition}

With this notion in place, we are ready to define the $2^d$-ary inner products of various degrees on $\mathbb{C}^G$ by averaging of multiplicative discrete derivatives over subspaces of higher degree cubes in $G^{\{0,1\} ^A}$. 

\begin{definition}\label{def:in_prod}
On the space of complex-valued functions on $G$, we define the {\it $2^d$-ary inner product of degree $l$ } as follows.  For $F=(f_\omega)$, a $d$-dimensional $A$-cube of  functions, set 
\begin{equation}\label{in_prod_G}
\left\langle F \right\rangle_{A,l} \coloneqq \mathbb{E}_{p \in \mathcal{P}^{A,l}(G)} \overline{\Delta}^A_{p,l} (F).
\end{equation}

Similarly, for $\widehat{F}$, an $A$-cube of complex-valued functions on the dual group $\widehat{G}$, define
\begin{equation}\label{in_prod_hatG}
\left\langle \widehat{F} \right\rangle_{A,l} \coloneqq \sum_{P \in \mathcal{P}^{A,l}(\widehat{G})} \overline{\Delta}^A_{P,l} (\widehat{F}).
\end{equation}
\end{definition}

These inner products are complex $2^{d-1}$-linear in even-weighted arguments and complex $2^{d-1}$-anti-linear in odd-weighted arguments. The $S_A$ part of the symmetry group $S_A \ltimes (\mathbb{Z}/2\mathbb{Z})^A$ of the cube $\{0,1\}^A$ clearly preserves $\left\langle \cdot \right\rangle_{A,l}$, and as we will see in Corollary \ref{cor:Fsplit}, the $(\mathbb{Z}/2\mathbb{Z})^A$ part is responsible for complex conjugate symmetry, similar to the standard Hermitian form on $\mathbb{C}^G$. 
Provided that these inner products are positive-definite, as we will also show in Corollary \ref{cor:Fsplit}, we will refer to them as $2^d$-Hermitian (for lack of a better term).  Tao, in Section 2.2. of \cite{Tao12}, uses the term \textit{quartisesquilinear} for $2^d$-ary inner product in the case $d=2$.

\begin{example}
 For  $(d,l)=(1,0)$, Definition \ref{def:in_prod} represents
the standard sesquilinear inner products on $\mathbb{C}^G$ and $\mathbb{C}^{\widehat{G}}$.

The two degenerate cases $(d,l)=(0,-1)$ and $(d,l)=(0,0)$ are also instructive. Unary inner products  $\left\langle \cdot \right\rangle_{\emptyset,-1}$,   $\left\langle \cdot \right\rangle_{\emptyset,0}$ are linear functionals on $\mathbb{C}^G$ (or $\mathbb{C}^{\widehat{G}}$). Although they are not even semi-definite, they lead to familiar identities:
$$ \left\langle f \right\rangle_{\emptyset,-1} = f(0) = \sum_{\chi \in \widehat{G}} \widehat{f}(\chi) =  \left\langle \widehat{f} \right\rangle_{\emptyset,0} $$
and
$$ \left\langle f \right\rangle_{\emptyset,0} = \mathbb{E}_{x \in G} f(x) = \widehat{f}(1) =  \left\langle \widehat{f} \right\rangle_{\emptyset,-1} .$$

For $(d,l)=(d,1)$ these definitions coincide with the
Gowers inner product $\left\langle \cdot \right\rangle_{U^d(G)}$ of order $d$, see, e.g. \cite{Tao12}, Section 1.3., since
\begin{multline*}
\left\langle (f_\omega)_{\omega \in  \{0,1\}^d} \right\rangle_{U^d(G)} = \mathbb{E}_{x,h_1, \dots , h_d \in G} \prod_{(\omega_1, \dots \omega_d) \in  \{0,1\}^d} \\
\mathcal{C}^{\omega_1 + \dots + \omega_d} f_{(\omega_1, \dots, \omega_d)} (x + \omega_1 h_1 + \dots + \omega_d h_d),\end{multline*} 
becomes, in our notation,
$$ \langle F \rangle_{d,1} = \mathbb{E}_{p \in \mathcal{P}^{\{1, \dots, d\},1}(G)} \overline{\Delta}^{\{1, \dots, d\}}_{p,1} (F).$$ 
Note as well that we can split the multiplicative discrete derivative operator $\overline{\Delta}^{\{1, \dots, d\}}_{p,1}$  at point $p = \left(x_{0\dots0} + \sum_i \omega_i h_i\right)_{\omega \in  \{0,1\}^d} \in \mathcal{P}^{\{1, \dots, d\},1}(G)$ as a $d$-fold composition of linear first order derivative operators:
$$ \overline{\Delta}^{\{1, \dots, d\}}_{p,1} (F) = \overline{\Delta}^{\{1\}}_{h_1} \overline{\Delta}^{\{2\}}_{h_2} \dots \overline{\Delta}^{\{d\}}_{h_d} (F)(x_{0 \dots 0}).$$
\end{example}

For discrete derivatives of higher degrees ($l \ge 2$), the point $p \in \mathcal{P}^{A,l}(G)$ no longer represents a $d$-dimensional parallelepiped, but the operator $\overline{\Delta}^A_{p,l}$  still can be  split as indicated in the next proposition where we list several additional properties of the multiplicative discrete derivative. 

\begin{proposition} 
\label{prop:dersplit}
Let $F=(F_{0_i}, F_{1_i})$ be a $2^d$-tuple of functions $f_{(\omega', \omega_i)}: G \to \mathbb{C}$ where $\omega' \in \{0,1\}^{A \setminus \{i\}}$. Then for  points  $p, q  \in \mathcal{P}^{A,l}(G)$, with $p=(p_{0_i}, p_{1_i}) =  (r_{0_i}+r_{d_i},r_{1_i}+r_{d_i})$ as in \eqref{cornersplit}, the discrete derivative $\overline{\Delta}^A_{p,l}$ satisfies:

\renewcommand{\labelenumi}{\roman{enumi})}
\begin{enumerate}
    
\item $\displaystyle \overline{\Delta}^A_{p,l} (F_{0_i}, F_{1_i}) = \overline{\Delta}^{A\setminus \{i\}}_{p_{0_i},l}\left( F_{0_i} \overline{T^{- \delta_i p}(F_{1_i})}\right)$,

\item $\displaystyle\overline{\Delta}^A_{p+q,l} (F) = \overline{\Delta}^A_{p,l} (T^qF)$,

\item $\displaystyle\overline{\Delta}^A_{(p_{1_i},p_{0_i}),l} (F_{0_i}, F_{1_i}) = \overline{\Delta}^A_{(p_{0_i},p_{1_i}),l} (T^{-\delta_i p}F_{0_i}, T^{\delta_i p}F_{1_i})$,

\item
$\displaystyle\overline{\Delta}^A_{(p_{0_i},p_{1_i}),l} (F_{1_i}, F_{0_i}) = \overline{\overline{\Delta}^A_{(p_{1_i},p_{0_i}),l} (F_{0_i}, F_{1_i}) } ,$

\item $\displaystyle \overline{\Delta}^A_{(p_{0_i},p_{1_i}),l} (F_{0_i}, F_{1_i}) =
\overline{\Delta}^{A\setminus \{i\}}_{r_{0_i},l-1}\left( T^{r_{d_i}}F_{0_i}\right) \overline{
\overline{\Delta}^{A\setminus \{i\}}_{r_{1_i},l-1}\left( T^{r_{d_i}}F_{1_i}\right)}.$
\end{enumerate}
\end{proposition}

\begin{proof}
For the first identity, recall that the point  $p=(p_{0_i},p_{1_i})$ in  $\mathcal{P}^{A,l}(G)$ can be identified with a pair 
$[\delta_i p , p \circ s_i^1 ] = [p_{0_i}-p_{1_i}, p_{1_i}]\in \mathcal{P}^{A \setminus \{i\},l-1}(G) \times   \mathcal{P}^{A \setminus \{i\},l}(G)$, where $p_{1_i} = p \circ s_i^1 = p \circ s_i^0 - \delta_i p = p_{0_i} - (p_{0_i} - p_{1_i})$ as in \eqref{Pdecomp}.

Then we have $$\overline{\Delta}^A_{p,l} (F) = \prod_{\omega \in \{0,1\}^A} \mathcal{C}^{|\omega|} f_\omega (x_\omega) = $$
$$ = \prod_{\omega' \in \{0,1\}^{A \setminus \{i\}}} \prod_{\omega_i \in \{0,1\}^{\{i\}}} \mathcal{C}^{|\omega'|+|\omega_i|} f_{(\omega', \omega_i)} (x_{(\omega', \omega_i )}) = $$
$$ = 
 \prod_{\omega' \in \{0,1\}^{A \setminus \{i\}}}  \mathcal{C}^{|\omega'|} \left[ f_{(\omega', 0_i)} (x_{(\omega', 0_i )}) \overline{f_{(\omega', 1_i)} (x_{(\omega', 0_i )} - {\delta_i p}_{\omega'} )} \right]= $$
$$ = 
 \prod_{\omega' \in \{0,1\}^{A \setminus \{i\}}}  \mathcal{C}^{|\omega'|} \left[ f_{(\omega', 0_i)}  \overline{T^{{-\delta_i p}_{\omega'}}f_{(\omega', 1_i)}} \right](x_{(\omega', 0_i )})= $$
 $$ = 
 \prod_{\omega' \in \{0,1\}^{A \setminus \{i\}}}  \mathcal{C}^{|\omega'|} \left[ \overline{\Delta}^{\{i\}}_{{-\delta_ip}_{\omega'}} \left(f_{(\omega', 0_i)} , f_{(\omega', 1_i)}\right) \right](x_{(\omega', 0_i )})= $$
 $$ = \overline{\Delta}^{A\setminus \{i\}}_{p_{0_i},l}\left( \overline{\Delta}^{\{i\}}_{{-\delta_ip}_{\omega'}} (f_{(\omega', 0_i)}, f_{(\omega', 1_i)})\right) = \overline{\Delta}^{A\setminus \{i\}}_{p_{0_i},l}\left( F_{0_i}\overline{T^{- \delta_i p}F_{1_i}}\right).$$
For the second identity, let $p=(x_\omega)$ and $q=(y_\omega)$ be points in $\mathcal{P}^{A,l}(G)$. Then  
$$\overline{\Delta}^A_{p+q,l} (F) = \prod_{\omega \in \{0,1\}^A} \mathcal{C}^{|\omega|} f_\omega (x_\omega + y_\omega) = 
\prod_{\omega \in \{0,1\}^A} \mathcal{C}^{|\omega|} T^{y_\omega}f_\omega (x_\omega) = 
\overline{\Delta}^A_{p,l} (T^qF).$$ 
For the third identity, note that $(p_{1_i},p_{0_i}) = (p_{0_i},p_{1_i}) + (-\delta_i p, \delta_i p)$. By applying the second identity ii), the claim follows.  

For the fourth identity, switching the order in the $i$-direction in the function-cube $ (F_{1_i}, F_{0_i})$ gives:

  $$\overline{\Delta}^A_{(p_{0_i},p_{1_i}),l} (F_{1_i}, F_{0_i}) =  \prod_{\omega' \in \{0,1\}^{A \setminus \{i\}}}  \mathcal{C}^{|\omega'|} \left[ f_{(\omega', 1_i)} (x_{(\omega', 0_i )}) \overline{f_{(\omega', 0_i)} (x_{(\omega', 1_i )} )} \right]= $$
 $$ 
  =  \prod_{\omega' \in \{0,1\}^{A \setminus \{i\}}}  \mathcal{C}^{|\omega'|+1} \left[ f_{(\omega', 0_i)} (x_{(\omega', 1_i )}) \overline{f_{(\omega', 1_i)} (x_{(\omega', 0_i )})} \right]= \overline{\overline{\Delta}^A_{(p_{1_i},p_{0_i}),l} (F_{0_i}, F_{1_i})}.
 $$
 
 For the fifth identity, recall the decomposition $p=(p_{0_i}, p_{1_i}) = (x_\omega)= (r_{0_i}+r_{d_i},r_{1_i}+r_{d_i})$, where $[r_{0_i}, r_{1_i},r_{d_i}] \in \mathcal{P}^{A\setminus \{i\},l-1}(G) \times \mathcal{P}^{A\setminus \{i\},l-1}(G) \times \mathcal{D}_{l-1}\mathcal{P}^{A \setminus \{i\},l}(G)$, as in \eqref{Pdecomp2} and \eqref{cornersplit}.
 
 Then, 
 $$\overline{\Delta}^A_{(p_{0_i},p_{1_i}),l} (F_{0_i}, F_{1_i})= \prod_{\omega \in \{0,1\}^A} \mathcal{C}^{|\omega|} f_\omega (x_\omega) = $$
 $$ = \prod_{\omega' \in \{0,1\}^{A \setminus \{i\}}} \prod_{\omega_i \in \{0,1\}^{\{i\}}} \mathcal{C}^{|\omega'|+|\omega_i|} f_{(\omega', \omega_i)} (r_{\omega_i}(\omega') + r_{d_i}(\omega')) = $$
  $$ = \prod_{\omega_i \in \{0,1\}^{\{i\}}}  \mathcal{C}^{|\omega_i|} \prod_{\omega' \in \{0,1\}^{A \setminus \{i\}}}  \mathcal{C}^{|\omega'|} T^{r_{d_i}(\omega')}f_{(\omega', \omega_i)} (r_{\omega_i}(\omega') ) = $$
$$=  \prod_{\omega_i \in \{0,1\}^{\{i\}}}  \mathcal{C}^{|\omega_i|} 
\overline{\Delta}^{A \setminus \{i\}}_{r_{\omega_i},l-1} T^{r_{d_i}}F_{\omega_i} 
 = \overline{\Delta}^{A\setminus \{i\}}_{r_{0_i},l-1}\left( T^{r_{d_i}}F_{0_i}\right)  \overline{
\overline{\Delta}^{A\setminus \{i\}}_{r_{1_i},l-1}\left( T^{r_{d_i}}F_{1_i}\right)}.
 $$ 
\end{proof}

\begin{corollary} 
\label{cor:Fsplit}
Let $(F_{0_i}, F_{1_i})$ be a $2^d$-tuple of functions $f_{(\omega', \omega_i)}: G \to \mathbb{C}$ where $\omega' \in \{0,1\}^{A \setminus \{i\}}$.
Then the inner product $\left\langle (F_{0_i}, F_{1_i}) \right\rangle_{A,l}$ satisfies:

\renewcommand{\labelenumi}{\roman{enumi})}
\begin{enumerate}
\item
$\displaystyle \left\langle \left(F_{0_i}, F_{1_i}\right) \right\rangle_{A,l} =  \mathbb{E}_{p' \in \mathcal{P}^{A \setminus \{i\},l-1}(G)} \left\langle   F_{0_i} \overline{T^{p'}F_{1_i}}  \right\rangle_{A\setminus\{i\},l}, $
\item
$ \displaystyle \left\langle \left(F_{1_i}, F_{0_i}\right) \right\rangle_{A,l} =  \overline{\left\langle \left(F_{0_i}, F_{1_i}\right) \right\rangle_{A,l}},$
\item $\displaystyle \left\langle \left(F_{0_i}, F_{1_i}\right) \right\rangle_{A,l} =  \mathbb{E}_{r \in \mathcal{D}_{l-1}\mathcal{P}^{A \setminus \{i\},l}(G)} \left[ \left\langle   T^r F_{0_i} \right\rangle_{A\setminus\{i\},l-1}  \overline{ \left\langle   T^r F_{1_i}  \right\rangle_{A\setminus\{i\},l-1}}\right],$

\item if $F_{0_i}= F_{1_i}$, the inner product  $\left\langle \left(F_{0_i}, F_{0_i}\right) \right\rangle_{A,l}$ is real and non-negative,

\item for $F = \left( f \right)_{\omega \in \{0,1\} ^A}$ and $0 \le l \le d-1$ the inner product $\left\langle F \right\rangle_{A,l}$ is positive definite.
\end{enumerate}
\end{corollary}

\begin{proof} For the first identity, we the decomposition of point-cubes $p \in \mathcal{P}^{A,l}(G)$ into pairs
$(p \circ s_i^1, \delta_i p) \in \mathcal{P}^{A \setminus \{i\},l}(G) \times \mathcal{P}^{A \setminus \{i\},l-1}(G)$ as in Proposition~\ref{prop:dersplit}(i). Since $\delta_i p \in \mathcal{P}^{A\setminus\{i\},l-1}(G)$, we can average over this subspace and perform a change of variables $p' = -\delta_i p$ to obtain:
$$ \left\langle \left(F_{0_i}, F_{1_i}\right) \right\rangle_{A,l} =  \mathbb{E}_{p \in \mathcal{P}^{A,l}(G)} \overline{\Delta}^A_{p,l} \left(F_{0_i}, F_{1_i}\right)= $$ 
$$= \mathbb{E}_{[p \circ s_i^1   ,\delta_i p] \in \mathcal{P}^{A\setminus\{i\},l}(G) \times \mathcal{P}^{A\setminus\{i\},l-1}(G)}   \left[ \overline{\Delta}^{A \setminus \{i\}}_{p \circ s_i^1  ,l} \left( F_{0_i}\overline{T^{-\delta_i p}F_{1_i}}\right) \right] = 
$$
$$= \mathbb{E}_{p'  \in \mathcal{P}^{A\setminus\{i\},l-1}(G)} \left[ \mathbb{E}_{p \circ s_i^1   \in \mathcal{P}^{A\setminus\{i\},l}(G)}     \overline{\Delta}^{A \setminus \{i\}}_{p \circ s_i^1  ,l}  \left( F_{0_i}\overline{T^{p'}F_{1_i}}\right) \right] = $$
$$ = \mathbb{E}_{p'  \in \mathcal{P}^{A\setminus\{i\},l-1}(G)}  \left\langle F_{0_i}\overline{T^{p'}F_{1_i}} \right\rangle_{A\setminus \{i\},l} .$$

For the second identity, we apply Proposition \ref{prop:dersplit} iv) and the fact, that the space $\mathcal{P}^{A,l}(G)$ is invariant under the involution switching the $i$-th coordinate in a point-cube
$(p_{0_i},p_{1_i}) \mapsto (p_{1_i},p_{0_i})$. We conclude:

$$\left\langle  \left(F_{1_i}, F_{0_i}\right)  \right\rangle_{A,l} = \mathbb{E}_{(p_{0_i},p_{1_i}) \in \mathcal{P}^{A,l}(G)} \overline{\Delta}^A_{(p_{0_i},p_{1_i}),l}  \left(F_{1_i}, F_{0_i}\right)  = $$
$$ = \mathbb{E}_{(p_{0_i},p_{1_i}) \in \mathcal{P}^{A,l}(G)} \overline{\overline{\Delta}^A_{(p_{1_i},p_{0_i}),l}  \left(F_{0_i}, F_{1_i}\right)} = \overline{\left\langle \left(F_{0_i}, F_{1_i}\right) \right\rangle_{A,l}}.$$

Claim  iii) follows from Proposition \ref{prop:dersplit} v) and the decomposition of the space of  degree-$l$ cubes $\mathcal{P}^{A,l}(G) \cong  \mathcal{P}^{A\setminus \{i\},l-1}(G) \times \mathcal{P}^{A\setminus \{i\},l-1}(G) \times \mathcal{D}_{l-1}\mathcal{P}^{A \setminus \{i\},l}(G)$, as given by \eqref{Pdecomp2}.

$$ \left\langle \left(F_{0_i}, F_{1_i}\right) \right\rangle_{A,l} =  \mathbb{E}_{(p_{0_i}, p_{1_i}) \in \mathcal{P}^{A,l}(G)} \overline{\Delta}^A_{(p_{0_i}, p_{1_i}),l} \left(F_{0_i}, F_{1_i}\right)= $$ 
\begin{multline*}
= \mathbb{E}_{[r_{0_i},r_{1_i},r_{d_i}] \in  \mathcal{P}^{A\setminus \{i\},l-1}(G) \times \mathcal{P}^{A\setminus \{i\},l-1}(G) \times \mathcal{D}_{l-1}\mathcal{P}^{A \setminus \{i\},l}(G) } \\
\left[\overline{\Delta}^{A\setminus \{i\}}_{r_{0_i},l-1}\left( T^{r_{d_i}}F_{0_i}\right)  \overline{
\overline{\Delta}^{A\setminus \{i\}}_{r_{1_i},l-1}\left( T^{r_{d_i}}F_{1_i}\right)}\right] = \end{multline*}

\begin{multline*}
= \mathbb{E}_{r_{d_i} \in  \mathcal{D}_{l-1}\mathcal{P}^{A \setminus \{i\},l}(G) } \bigg[ \left(\mathbb{E}_{r_{0_i} \in  \mathcal{P}^{A \setminus \{i\},l-1}(G)} \overline{\Delta}^{A\setminus \{i\}}_{r_{0_i},l-1}\left( T^{r_{d_i}}F_{0_i}\right)\right) \cdot \\ \cdot
\overline{ \left(\mathbb{E}_{r_{1_i} \in  \mathcal{P}^{A \setminus \{i\},l-1}(G)}  \overline{\Delta}^{A\setminus \{i\}}_{r_{1_i},l-1}\left( T^{r_{d_i}}F_{1_i}\right)\right)} \bigg]=
\end{multline*}

$$=\mathbb{E}_{r \in \mathcal{D}_{l-1}\mathcal{P}^{A \setminus \{i\},l}(G)} \left[ \left\langle   T^{r}F_{0_i} \right\rangle_{A\setminus\{i\},l-1}  \overline{ \left\langle   T^{r}F_{1_i}  \right\rangle_{A\setminus\{i\},l-1}}\right],$$
where we relabeled variable $r_{d_i}$ as $r$ in the final step.

The fact that  $\left\langle \left(F_{0_i}, F_{0_i}\right) \right\rangle_{A,l}$ is real and non-negative follows directly from parts ii) and iii). 

For $F = \left( f \right)_{\omega'' \in \{0,1\} ^{A'}}$ and any non-empty $A'\subseteq A$, with $|A'|=k$, we observe that
$$\left\langle F \right\rangle_{A',0} = \sum_{x \in G} \left|f(x)\right|^{2^k}  =  \| f\|^{2^k}_{L^{2^k}(G)},$$ 
since the space $\mathcal{P}^{A',0}$ of degree-$0$ $A'$-cubes consists  only of diagonal elements in $G^{2^k}$. Hence, the $2^k$-ary inner product $\left\langle \cdot \right\rangle_{A',0}$ is positive definite.

Moreover, the space $\mathcal{D}_{l-1}\mathcal{P}^{A \setminus \{i\},l}(G)$ contains the zero element $r_{d_i}=0$. In this case,  $T^{r_{d_i}}F_{0_i} = F_{0_i} = \left( f\right)_{\omega' \in \{0,1\} ^{A\setminus\{i\}}}$. By induction on $|A|$, we may assume that the inner product 
$\langle \left( f\right)_{\omega' \in \{0,1\} ^{A\setminus\{i\}}} \rangle_{A\setminus\{i\},l-1}$ is positive whenever $f \neq 0 \in \mathbb{C}^G$. Then, applying part  iii), we deduce that  $\langle \cdot \rangle_{A,l}$ is also positive definite.

We note that this inductive argument relies on the averaging over a finite group $G$, and therefore applies only when $G$ is finite.  
\end{proof}

\begin{remark} As a special case, for $\left( f \right)_{\omega \in \{0,1\} ^A}$ and $l=1$, part i) recovers the recursive formula for Gowers uniformity norms $U^A$:
$$ \|f\|_{U^A(G)}^{2^d} = 
\left\langle (f)_{\omega \in  \{0,1\}^A} \right\rangle_{A,1}  
=$$
$$ = \mathbb{E}_{p'  \in \mathcal{P}^{A\setminus\{i\},0}(G)}  \left\langle  \left(f\right)_{\omega' \in  \{0,1\}^{A\setminus\{i\} }}  \overline{T^{p'}\left(f\right)_{\omega' \in  \{0,1\}^{A\setminus\{i\} }}} \right\rangle_{A\setminus \{i\},1}= $$
$$ = \mathbb{E}_{p'  \in \mathcal{P}^{A\setminus\{i\},0}(G)} \left\langle \left(\overline{\Delta}^{\{i\}}_{p'_{\omega'}} (f,f) \right)_{\omega' \in  \{0,1\}^{A\setminus\{i\} }}
\right\rangle_{A\setminus \{i\},1} = 
\mathbb{E}_{h \in G} \left\|\Delta_{h} f \right\|^{2^{d-1}}_{U^{A\setminus\{i\}}(G)},$$
since each $p'  \in \mathcal{P}^{A\setminus\{i\},0}(G)$ is a constant $2^{(d-1)}$-tuple $(h)_{\omega' \in  \{0,1\}^{A\setminus\{i\} }}$ for some $h \in G$.

Part iii) is analogous to the \emph{splitting axiom} used in the definition of a higher-order inner product space, see 
\cite[Definition~2.2.1]{Tao12}. 

%
\end{remark}

Parts iv) and v) of the previous corollary motivate the following definition. 

\begin{definition}
Let $f: G \to \mathbb{C}$ be a function, and consider the $A$-cube $(f)_{\omega \in \{0,1\}^A}$. For any $l \ge 0$, we define the \textit{uniformity norm of order $d $ and degree $l$} by
$$
\|f\|_{U^{A,l}} \coloneqq \left\langle \left(f\right)_{\omega \in \{0,1\}^A} \right\rangle_{A,l}^{1/{2^d}}.
$$
As we will see in the next two propositions, this defines a (semi-)norm on the space $\mathbb{C}^G$ of complex-valued functions on $G$.
\end{definition}

\begin{proposition}[Cauchy-Schwarz-Gowers inequality]
\label{CSG}
Let $F = (f_\omega)_{\omega \in \{0,1\}^A}$ be a $2^d$-tuple of functions $f_\omega: G \to \mathbb{C}$.
Then the inner product $\left\langle F \right\rangle_{A,l}$ satisfies the Cauchy-Schwarz-Gowers inequality:
$$ \left| \left\langle F  \right\rangle_{A,l} \right| \le \prod_{\omega \in \{0,1\}^{A}}  \left\| f_\omega  \right\|_{U^{A,l}}. $$
\end{proposition} 
\begin{proof} Fix $i \in A$, and split the $A$-cube $F$ into a pair $\left(F_{0_i}, F_{1_i}\right)$ of $A \backslash \{i\}$-cubes. By part  iii) of Corollary \ref{cor:Fsplit}, and then using the classical Cauchy-Schwarz inequality over the  averaging space $ \mathcal{D}_{l-1}\mathcal{P}^{A \setminus \{i\},l}(G)$, we obtain
$$ \left| \left\langle F  \right\rangle_{A,l} \right| = \left| \left\langle \left(F_{0_i}, F_{1_i}\right)  \right\rangle_{A,l} \right| = $$
$$= \left| \mathbb{E}_{r \in \mathcal{D}_{l-1}\mathcal{P}^{A \setminus \{i\},l}(G)} \left[ \left\langle   T^{r}F_{0_i} \right\rangle_{A\setminus\{i\},l-1}  \overline{ \left\langle   T^{r}F_{1_i}  \right\rangle_{A\setminus\{i\},l-1}}\right] \right| \le
$$ 
\begin{multline*}
\le \left(\mathbb{E}_{r \in \mathcal{D}_{l-1}\mathcal{P}^{A \setminus \{i\},l}(G)} \left| \left\langle   T^{r}F_{0_i} \right\rangle_{A\setminus\{i\},l-1}\right|^2\right)^{\frac{1}{2}} \cdot \\ \cdot
\left(\mathbb{E}_{r \in \mathcal{D}_{l-1}\mathcal{P}^{A \setminus \{i\},l}(G)} \left| \left\langle   T^{r}F_{1_i} \right\rangle_{A\setminus\{i\},l-1}\right|^2\right)^{\frac{1}{2}}=
\end{multline*}
$$ = \left\langle \left(F_{0_i}, F_{0_i}\right)  \right\rangle_{A,l}^{\frac{1}{2}} \left\langle \left(F_{1_i}, F_{1_i}\right)  \right\rangle_{A,l}^{\frac{1}{2}}.$$

 Iterating this decomposition over all coordinates \(i \in A\) yields the desired bound
$$
\left| \left\langle F \right\rangle_{A,l} \right| \le \prod_{\omega \in \{0,1\}^A} \left\| f_\omega \right\|_{U^{A,l}}.
$$
\end{proof}

\begin{proposition}[Minkowski inequality]  Let $f_0,f_1 : G \to \mathbb{C}$ be two functions. Then $ \|f_0+f_1\|_{U^{A,l}} \le  \|f_0\|_{U^{A,l}} +  \|f_1\|_{U^{A,l}}$.  
\end{proposition}
\begin{proof}
This follows from an argument originally given by  Gowers in \cite{Gow01}. We expand   the left-hand side   by  multilinearity of the $2^d$-ary inner product and then apply the Cauchy-Schwarz-Gowers inequality to each the term to get
$$\left\langle (f_0+ f_1)_{\omega \in \{0,1\}^A}  \right\rangle_{A,l} = \left|\sum_{\epsilon \in \{0,1\}^{\{0,1\}^A}} \left\langle (f_{\epsilon(\omega)})_{\omega \in \{0,1\}^A}  \right\rangle_{A,l}\right|= $$
$$=\left|\sum_{k=0}^{2^d}\sum_{|\epsilon|=2^d-k} \left\langle (f_{\epsilon(\omega)})_{\omega \in \{0,1\}^A}  \right\rangle_{A,l}\right|\le \sum_{k=0}^{2^d}\sum_{|\epsilon|=2^d-k} \|f_0\|_{U^{A,l}}^{k}\|f_1\|_{U^{A,l}}^{2^d -k}=$$
$$
= \sum_{k=0}^{2^d} \binom{2^d}{k} \|f_0\|_{U^{A,l}}^{k}\|f_1\|_{U^{A,l}}^{2^d -k}
=( \|f_0\|_{U^{A,l}}+\|f_1\|_{U^{A,l}})^{2^d}.$$
\end{proof}

As a straightforward consequence of the Cauchy-Schwarz-Gowers inequality, we get the following inequality between uniformity norms of the same degree $l$.

\begin{proposition}
    Let  $f: G \to \mathbb{C}$ and $i \in A$. Then \[\|f\|_{U^{A,l}} \ge \|f\|_{U^{A \setminus \{i\},l}}.\]
\end{proposition} 
\begin{proof} Let $1:G \to \mathbb{C}$ denote the constant function $1$. Using part i) of of Corollary \ref{cor:Fsplit}, and applying Cauchy-Schwarz-Gowers inequality, we obtain 
\[ \|f\|_{U^{A \setminus \{i\},l}}^{2^{d-1}} = \mathbb{E}_{p' \in \mathcal{P}^{A \setminus \{i\},l-1}(G)} \|f\|_{U^{A \setminus \{i\},l}}^{2^{d-1}} = \]
\[=
\mathbb{E}_{p' \in \mathcal{P}^{A \setminus \{i\},l-1}(G)} \left\langle   (f)_{\omega' \in  \{0,1\}^{A\setminus\{i\} }} \overline{T^{p'}(1)_{\omega' \in  \{0,1\}^{A\setminus\{i\} }}}  \right\rangle_{A\setminus\{i\},l} = \]
\[  = 
\left\langle  \left(f\right)_{\omega' \in  \{0,1\}^{A\setminus\{i\} }}  ,\left(1\right)_{\omega' \in  \{0,1\}^{A\setminus\{i\} }} \right\rangle_{A,l} \le \|f\|_{U^{A ,l}}^{2^{d-1}} \|1\|_{U^{A,l}}^{2^{d-1}} = \|f\|_{U^{A ,l}}^{2^{d-1}}.
\]
\end{proof}

\begin{remark} The Gowers uniformity norm $U^{d+1}$ is maximized precisely by degree-$d$ phase polynomial functions. Specifically, for any $f:G \to \mathbb{C}$ with $|f(x)| \le 1$, we have $\|f\|_{U^{d+1}}=1$ if and only if $f(x)=e(P(x))$, where $P: G \to \mathbb{R}/\mathbb{Z}$ is a polynomial of degree at most $d$. 

Since the generalized uniformity norms \( U^{A,l} \) are computed by averaging over degree-\( l \) cubes in \( G \), it follows that \( \|f\|_{U^{A,l}} = 1 \) whenever \( f \) is a phase polynomial of degree at most \( d - l \). Hence, these norms are sensitive to and detect phase polynomials of degree up to $d - l$.
\end{remark}

\section{Orthogonality and Isometry} 
\label{sec:orto-fin}

We now focus on a fundamental special case: evaluating multiplicative discrete derivatives and the associated $2^d$-ary inner products on cubes of additive characters. Consider a map $(p,P) : \{0,1\}^A \to G \times \widehat{G}$, where $p=(x_\omega)\in G^{\{0,1\}^A}$ and $P=(\chi_\omega) \in {\hat{G}}^{\{0,1\}^A}$, with $d=|A|$

Recall that he multiplicative discrete derivative of order $d$ and degree $l$ is defined by 
$$ \overline{\Delta}^A_{p,l} (P)  =  \prod_{\omega \in \{0,1\}^A} \mathcal{C}^{|\omega|} \chi_\omega (x_\omega).$$

A key observation allows us to improve Claim i) in Proposition \ref{prop:dersplit} by expressing $\overline{\Delta}^A_{p,l} (P)$as a product of two multiplicative derivatives of one order lower. 

For arbitrary elements $a,b \in G$ and characters $\chi_\alpha, \chi_\beta \in \widehat{G}$, note that   
$$\chi_{\alpha}(a)\overline{\chi_\beta(b)} = \chi_{\alpha}(a)\overline{\chi_{\alpha}(b)}\chi_\alpha(b)\overline{\chi_\beta(b)} = \chi_{\alpha}(a-b)\left(\chi_{\alpha} - \chi_\beta\right)(b).$$
Similarly,  interchanging the roles of $a$, $b$ and $\alpha$,  $\beta$, yields  
$$\chi_{\alpha}(a)\overline{\chi_\beta(b)}=\left(\chi_\alpha-\chi_\beta\right)(a)\chi_\beta(a-b).$$
This computation, which appears in a related form in Lemma 2.14 of Section 2.7. in \cite{Szeg12}, underpins Proposition \ref{prop:charsplit}. Unlike the treatment there, we carry this idea  to its full logical conclusion  in  Corollary \ref{cor:split} and the Orthogonality lemma \ref{lem:orth}.

The above equality allows us to replace a pair of  point-character $A$-cubes $(p, P )$ with two pairs $( \delta_i p, P \circ s_i^0 )$ and $(p \circ s_i^1  , \delta_i P)$ of $A \backslash \{i\}$-cubes.

\begin{proposition}
\label{prop:charsplit}
Let $(p,P) : \{0,1\}^A \to G \times \widehat{G}$ be a point-character $A$-cube with \(p \in \mathcal{P}^{A,l}(G)\). Then for any $i \in A$, we have the decomposition
$$ \overline{\Delta}^A_{p,l} (P)  =  \overline{\Delta}^{A \setminus \{i\}}_{\delta_i p,l-1}(P \circ s_i^0 ) \, \overline{\Delta}^{A \setminus \{i\}}_{p \circ s_i^1 ,l} (\delta_i P). $$
\end{proposition}
\begin{proof} Since $p \in  \mathcal{P}^{A,l}(G)$, we know  $\delta_i p \in  \mathcal{P}^{A \setminus \{i\},l-1}(G)$ and $p \circ s_i^1  \in  \mathcal{P}^{A \setminus \{i\},l}(G)$, as established  in \eqref{Pdecomp}.

We use the natural identification $\{0,1\}^A = \{0,1\}^{\{i\} } \times \{0,1\}^{A \setminus \{ i\} }$, and rewrite the multiplicative discrete derivative as
$$  \overline{\Delta}^A_{p,l} (P) = \prod_{\omega \in \{0,1\}^A} \mathcal{C}^{|\omega|} \chi_\omega (x_\omega) = $$
$$ = \prod_{\omega' \in \{0,1\}^{A \setminus \{i\}}} \prod_{\omega_i \in \{0,1\}^{\{i\}}} \mathcal{C}^{|\omega'|+|\omega_i|} \chi_{(\omega', \omega_i)} (x_{(\omega', \omega_i )}) = $$
$$ = 
 \prod_{\omega' \in \{0,1\}^{A \setminus \{i\}}}  \mathcal{C}^{|\omega'|} \left[ \chi_{(\omega', 0_i)} (x_{(\omega', 0_i )}) \overline{\chi_{(\omega', 1_i)} (x_{(\omega', 1_i )})} \right]= $$
 $$= \prod_{\omega' \in \{0,1\}^{A \setminus \{i\}}}  \mathcal{C}^{|\omega'|} \left[ \chi_{(\omega', 0_i)} (x_{(\omega', 0_i)} - x_{(\omega', 1_i)} ) \left(\chi_{(\omega', 0_i)} - \chi_{(\omega', 1_i)}\right) (x_{(\omega', 1_i )}) \right]= 
$$
$$ = \left[\prod_{\omega' \in \{0,1\}^{A \setminus \{i\}}}  \mathcal{C}^{|\omega'|} \chi_{s_i^0(\omega')}((\delta_i p)_{\omega'})\right]  \left[\prod_{\omega' \in \{0,1\}^{A \setminus \{i\}}}  \mathcal{C}^{|\omega'|} (\delta_i\chi)_{\omega'}(x_{s_i^1 (\omega')})\right]=
$$
$$=   \overline{\Delta}^{A \setminus \{i\}}_{\delta_i p,l-1}(P \circ s_i^0  ) \, \overline{\Delta}^{A \setminus \{i\}}_{p \circ s_i^1  ,l} (\delta_i P).$$

\end{proof}

As a corollary, after averaging over $\mathcal{P}^{A,l}(G) \cong  \mathcal{P}^{A\setminus\{i\},l}(G) \times \mathcal{P}^{A\setminus\{i\},l-1}(G)$,  we obtain the following.

\begin{corollary}
\label{cor:split}
For a character $A$-cube  $P : \{0,1\}^A \to \widehat{G}$, the value of the inner product $\left\langle P \right\rangle_{A,l}$ can be expressed as
$$ \left\langle P \right\rangle_{A,l}  = 
\left\langle \delta_i P \right\rangle_{A\setminus \{i\},l} 
\left\langle P \circ s_i^0 \right\rangle_{A \setminus \{i\},l-1}.$$
\end{corollary}

\begin{proof}
Using the definition of the $2^d$-ary inner product of degree $l$, Proposition \ref{prop:charsplit}, and the decomposition $\mathcal{P}^{A,l}(G) \cong  \mathcal{P}^{A\setminus\{i\},l}(G) \times \mathcal{P}^{A\setminus\{i\},l-1}(G)$, we obtain

$$\left\langle P \right\rangle_{A,l} = \mathbb{E}_{p \in \mathcal{P}^{A,l}(G)} \overline{\Delta}^A_{p,l} (P) = $$ 
$$= \mathbb{E}_{[p \circ s_i^1   ,\delta_i p] \in \mathcal{P}^{A\setminus\{i\},l}(G) \times \mathcal{P}^{A\setminus\{i\},l-1}(G)}   \left[  \overline{\Delta}^{A \setminus \{i\}}_{p\circ s_i^1 ,l} (\delta_i P) \, \overline{\Delta}^{A \setminus \{i\}}_{\delta_i p,l-1}(P \circ s_i^0  ) \right] = 
$$
$$= \left[ \mathbb{E}_{p \circ s_i^1   \in \mathcal{P}^{A\setminus\{i\},l}(G)}     \overline{\Delta}^{A \setminus \{i\}}_{p \circ s_i^1  ,l} (\delta_i P)\right ]  \left[  \mathbb{E}_{\delta_i p  \in \mathcal{P}^{A\setminus\{i\},l-1}(G)}   \overline{\Delta}^{A \setminus \{i\}}_{\delta_i p,l-1}(P \circ s_i^0  ) \right] = $$
$$ = \left\langle \delta_i P \right\rangle_{A\setminus \{i\},l} \,
\left\langle P \circ s_i^0 \right\rangle_{A \setminus \{i\},l-1}.$$

\end{proof}

\begin{lemma}
\label{lem:orth}
 {\it (Orthogonality lemma)}
Let $P$ be an  $A$-cube of characters on $G$ of degree $k\in \{-1, 0, \dots, d \}$. Then $P$  satisfies the following orthonormality condition with respect to the $2^d$-ary inner product of degree $l\in \{-1, 0, \dots, d \}$:
$$ \langle P \rangle_{A,l} = \begin{cases} 1, & \text{if } k + l < d, \\ 0, & \text{if } k + l \ge d. \end{cases}$$
\end{lemma}

\begin{proof}
We will proceed by induction on the dimension  $d = |A|$. The case $A=\emptyset$ is left as an exercise, and we will discuss  the situation for 
 $d=1$ in more detail.  
 
Consider  the set of $1$-dimensional cubes in $\widehat{G}$, which are pairs of characters on $G$. This set is subject to a a filtration $\{I\}=\mathcal{P}^{A,-1}(\widehat{G}) \subset \mathcal{P}^{A,0}(\widehat{G}) \subset \mathcal{P}^{A,1}(\widehat{G}) = \widehat{G}^{\{0,1\}^{A}}$.

The set $\mathcal{P}^{A,-1}(\widehat{G})$ of cubes of degree $-1$ contains only one pair of trivial characters $(1_0,1_1)$. We can verify that for all possible choices of $l \in \{-1,0,1\}$, the claim holds:
$$ \langle \left(1_0,1_1\right) \rangle_{A,-1} = 1_0(0)\overline{1_1(0)} = 1,$$
$$ \langle \left(1_0,1_1\right) \rangle_{A,0} = \mathbb{E}_{x \in G}1_0(x)\overline{1_1(x)} = 1,$$
$$ \langle \left(1_0,1_1\right) \rangle_{A,1} = \mathbb{E}_{x_0,x_1 \in G}1_0(x_0)\overline{1_1(x_1)} = \mathbb{E}_{x \in G}1_0(x) \, \overline{\mathbb{E}_{x \in G}1_1(x)}= 1.$$

An $A$-cube $P$ is of degree 0  whenever it is a nontrivial constant map, i.e., $\chi_0 = \chi_1 \neq 1$. 
In this case, the claim follows from the fact that nontrivial characters have a unit $L_2$ norm and average to zero over $G$:
$$ \langle \left(\chi,\chi\right) \rangle_{A,-1} = \chi(0)\overline{\chi(0)} = 1,$$
$$ \langle \left(\chi, \chi\right) \rangle_{A,0} = \mathbb{E}_{x \in G}\chi(x)\overline{\chi(x)} = 1,$$
$$ \langle \left(\chi,\chi\right) \rangle_{A,1} = \mathbb{E}_{x_0,x_1 \in G}\chi(x_0)\overline{\chi(x_1)} = \mathbb{E}_{x \in G}\chi(x) \, \overline{\mathbb{E}_{x \in G}\chi(x)}= 0.$$

The case $k=1$ corresponds to pairs $P=(\chi_0, \chi_1)$ where $\chi_0 \neq \chi_1$. For such pairs, the claim follows from the orthogonality of characters in the standard Hermitian inner product on $\mathbb{C}^G$. Specifically, at least one character in a non-constant pair must average to zero over $G$:
$$\langle (\chi_0,\chi_1)\rangle_{A,-1} =  \chi_0(0)\overline{\chi_1(0) }= 1, $$
$$\langle (\chi_0,\chi_1)\rangle_{A,0} = \mathbb{E}_{x\in G} \chi_0(x)\overline{\chi_1(x) }= 0, $$
$$ \langle \left(\chi_0,\chi_1\right) \rangle_{A,1} = \mathbb{E}_{x_0,x_1 \in G}\chi_0(x_0)\overline{\chi_1(x_1)} = \mathbb{E}_{x \in G}\chi_0(x) \, \overline{\mathbb{E}_{x \in G}\chi_1(x)}= 0.$$

Similarly, for  $|A|=d>1$ and $l=-1$,  we  observe that $\langle P \rangle_{A,-1} =1$ for all $P \in \widehat{G}^{\{0,1\}^{A}}$. This follows because when $l = -1$, the multiplicative discrete derivative $\overline{\Delta}^A_{p,-1}(P)$ is evaluated solely at the $2^d$-dimensional zero point $p =0 \in \mathcal{P}^{A,-1}( G)$. 

On the other extreme, when $l=d$, we average over the entire space ${G}^{{\{0,1\}}^A}$ and can interchange the order of averaging and multiplication  
$$\langle P \rangle_{A,d} =
 \mathbb{E}_{p \in {G}^{{\{0,1\}}^A}}\prod_{\omega \in \{0,1\}^A} \mathcal{C}^{|\omega|} \chi_\omega (x_\omega) = 
\prod_{\omega \in \{0,1\}^A} \mathcal{C}^{|\omega|} \mathbb{E}_{x \in G}\chi_\omega (x).$$ 
This expression is nonzero if and only if $P$ is the $A$-cube of trivial characters  $I \in \mathcal{P}^{A,-1}(\widehat{G})$, i.e., has degree $k=-1$. In this case, the average equals $1$.

For $k=-1$, all multiplicative discrete derivatives of the trivial $A$-cube of characters $I$ are equal to one. Hence, after averaging over any $\mathcal{P}^{A,l}(G)$, we get  
$$\langle I \rangle_{A,l} = \mathbb{E}_{p \in  \mathcal{P}^{A,l}(G)} \overline{\Delta}^A_{p,l} (I) = 1.$$

Now, we may assume that $d\ge 2$, $k\in \{0,1, \dots d\} $,  $l\in \{0,1, \dots d-1\}$ and proceed with the induction step. 

Consider an $A$-cube of characters $P$ of degree $k$. Then, by Proposition \ref{prop:dBshift}, $\delta_i P$ is an $A\backslash  \{i\}$-cube of characters of degree $k-1$ and $ P \circ s_i^0$ is an $A\backslash \{i\}$-cube of characters of degree $k$. By induction hypothesis, we have
$$\left\langle \delta_i P \right\rangle_{A\setminus \{i\},l} = \begin{cases} 1, & \text{if } k-1 +l  < d -1, \\ 0, & \text{otherwise,} \end{cases},$$
and
$$\left\langle P \circ s_i^0 \right\rangle_{A\setminus \{i\},l-1} = \begin{cases} 1, & \text{if } k +l -1 < d -1, \\ 0, & \text{otherwise.} \end{cases}$$
Using Corollary \ref{cor:split}, it follows 
$$ \left\langle P \right\rangle_{A,l}  = 
\left\langle \delta_i P \right\rangle_{A\setminus \{i\},l} 
\left\langle P \circ s_i^0 \right\rangle_{A \setminus \{i\},l-1} =\begin{cases} 1, & \text{if } k + l  < d, \\ 0, & \text{otherwise.} \end{cases}
$$
\end{proof}

\begin{remark}
We note  that a variant of the orthogonality lemma also applies to infinite compact abelian groups, such as $\mathbb{R}/\mathbb{Z}$. The finiteness of $G$ is not essential in the proof of Proposition \ref{prop:charsplit}. In the proof of Corollary \ref{cor:split}, the key requirement is that the double integral --- taken as an average over the compact space $\mathcal{P}^{A,l}(G)$ --- factorizes into two single integrals over the compact spaces $\mathcal{P}^{A\setminus \{i\},l}(G)$ and 
$\mathcal{P}^{A\setminus \{i\},l-1}(G)$. This factorization holds because the multiplicative discrete derivative of a cube of additive characters is continuous and bounded by 1. Finally, Lemma \ref{lem:orth} relies on the fact that additive characters from $\widehat{G}$ form an orthonormal basis of the Hilbert space $L_2(G)$.
\end{remark}

\begin{theorem}[Proof of Identity \eqref{ParsId}]
\label{thm:f=fhat}
Let $F$ be an $A$-cube of complex-valued functions on $G$, $F: G^{\{0,1\}^A} \to \mathbb{C}$, and let $\widehat{F}$ be its Fourier transform, $\widehat{F}: \widehat{G}^{\{0,1\}^A} \to \mathbb{C}$, with $F=\left( f_\omega\right)_{\omega \in \{0,1\}^A}$ and $\widehat{F}=\left( \widehat{f}_\omega\right)_{\omega \in \{0,1\}^A}$, where  $\widehat{(f_\omega)} = \widehat{f}_\omega$ for all $\omega \in \{0,1\}^A$. Then the $2^d$-ary inner products on $\mathbb{C}^G$ and $\mathbb{C}^{\widehat{G}}$ of degrees $l$ and $d-l-1$ satisfy 
$$\left\langle F \right\rangle_{A,l} =  \left\langle \widehat{F}
\right\rangle_{A,d-l-1} .
$$
\end{theorem}
\begin{proof}
For every $\omega \in \{0,1\}^A$, express $f_\omega : G \to \mathbb{C}$ as $f_\omega =\sum_{\xi \in \widehat{G}} \widehat{f}_{\omega, \xi} \chi_\xi$. By the multilinearity of the $2^d$-ary inner product of degree $l$ for the $A$-cube of functions $F = \left( f_\omega\right)$,  and rearranging sums and averages, we have 
$$ \left\langle F \right\rangle_{A,l} = \mathbb{E}_{p \in \mathcal{P}^{A,l}(G)}  \prod_{\omega \in \{0,1\}^A} \mathcal{C}^{|\omega|} f_\omega (x_\omega) =$$
$$ = \mathbb{E}_{p \in \mathcal{P}^{A,l}(G)}  \prod_{\omega \in \{0,1\}^A} \mathcal{C}^{|\omega|} \left( \sum_{\xi \in \widehat{G}} \widehat{f}_{\omega, \xi} \chi_\xi(x_\omega) \right) =$$
$$= \mathbb{E}_{p \in \mathcal{P}^{A,l}(G)} \sum_{ P \in {\widehat{G}}^{\{0,1\}^A}}   \prod_{\omega \in \{0,1\}^A} \mathcal{C}^{|\omega|} \widehat{f}_{\omega, P_\omega} \chi_\omega(x_\omega) =$$
$$=\sum_{P \in {\widehat{G}}^{\{0,1\}^A}} \mathbb{E}_{p \in \mathcal{P}^{A,l}(G)}
\prod_{\omega \in \{0,1\}^A} \mathcal{C}^{|\omega|} \widehat{f}_{\omega, P_\omega} \chi_\omega(x_\omega)=
$$
$$=\sum_{P \in {\widehat{G}}^{\{0,1\}^A}} 
\left(\prod_{\omega \in \{0,1\}^A} \mathcal{C}^{|\omega|} \widehat{f}_{\omega, P_\omega}\right)\mathbb{E}_{p \in \mathcal{P}^{A,l}(G)} \prod_{\omega \in \{0,1\}^A} \mathcal{C}^{|\omega|} \chi_\omega(x_\omega)=
$$
$$ \sum_{P \in {\widehat{G}}^{\{0,1\}^A}} \left(\prod_{\omega \in \{0,1\}^A} \mathcal{C}^{|\omega|}\widehat{f}_{\omega, P_\omega}\right)  \left\langle P \right\rangle_{A,l}  \overset{\substack{\text{Lemma} \\ \text{\ref{lem:orth}}}}= 
\sum_{P \in \mathcal{P}^{A,d-l-1}(\widehat{G})} \left(\prod_{\omega \in \{0,1\}^A} \mathcal{C}^{|\omega|}\widehat{f}_{\omega, P_\omega}\right)  $$
$$ = \sum_{P \in \mathcal{P}^{A,d-l-1}(\widehat{G})} \overline{\Delta}^A_{P, d-l-1} (\widehat{F}) =
 \left\langle \widehat{F} \right\rangle_{A,d-l-1}.$$
\end{proof}

\begin{corollary} For a finite abelian group $G$, the Fourier transform defines an isometry between the normed spaces  $(\mathbb{C}^G, \|\cdot\|_{U^{d,l}(G)})$ and  $(\mathbb{C}^{\widehat{G}},\|\cdot\|_{U^{d,d-l-1}(\widehat{G})})$.    \qed 
\end{corollary}

\begin{remark}
    As a consequence of the identity relating the norms $U^{d,d-1}(G)$ and $U^{d,0}(\hat G)$ --- where the latter corresponds to the standard $L^{2^d}$-norm on $C^{\hat{G}}$ --- we obtain the following monotonicity of norms:
$$ \|f\|_{U^{d,d-1}} \ge \|f\|_{U^{d+1,d}}.$$ 
In fact, a more general inequality holds: 
$$ \|f\|_{U^{d,l-1}} \ge \|f\|_{U^{d+1,l}} \quad \text{for } 1 \le l \le  d-1, $$ 
and its proof will appear in a subsequent paper.
        
\end{remark}

\section{Higher Order Poisson Summation Formula and Corner Convolutions} 
\label{sec:Poisson}

In this section, we outline an alternative, and arguably simpler, approach to proving Theorem \ref{thm:f=fhat}.

In recent literature on Brascamp-Lieb inequalities, one frequently  encounters multilinear integral formulas of the form 

\[ \int_H f_1 \otimes \dots \otimes f_m = \int_{H^\perp} \widehat{f}_1 \otimes \dots \otimes \widehat{f}_m,\]
where each $f_j : \mathbb{R}^{n_j} \to \mathbb{C}$ is sufficiently nice integrable functions and $H^\perp$ denotes the orthogonal complement of a subspace $H$ in $\mathbb{R}^{n_1} \times \dots \times \mathbb{R}^{n_m}$. 

This identity is referred to variously as \textit{Fourier invariance property} \cite{BenJeon20, BenBez22}, a \textit{generalized form of Parseval's identity} \cite{BenBez20}, or simply used without a special name as in \cite{DurThie21}, where a translated version is also discussed. In particular, \cite{BenJeon20} derives the isometry between the $U^{d,1}(G)$- and $U^{d,d-2}(\widehat{G})$-norms as a special case of this invariance for a finite abelian group $G$ and its dual $\widehat{G}$. 

In our context, we require minor modifications to incorporate complex conjugation and to appropriately adapt the notion of orthogonality, yielding what we call the \textit{higher order Poisson summation formula} (see \eqref{Poisson}). 

First, some additional notation:

Let $G$ be a finite abelian group, and let $B$ a finite index set with $|B|=k$. Fix a \textit{signature map} $\mathrm{sign}:B \to \{0,1\}$. Then $G^B$ is a finite abelian group consisting of elements $(x_j)_{j\in B}$, and its Pontryagin dual $\widehat{G^B}$ is isomorphic to ${(\widehat{G})^B}$, with elements $(\chi_j)_{j\in B}$. 

Given a subgroup of $H \le G^B$ define its {\it dual subgroup in ${\widehat{G}}^B$ with respect to signature} $\mathrm{sign}$ as 
$$ H^{\perp_{\mathrm{sign}}} \coloneqq \left\{ (\chi_j)_{j\in B} \in (\widehat{G})^B \middle| \prod_{j \in B}  \mathcal{C}^{\mathrm{sign}(j)} \chi_j(x_j) = 1 \, \text{ for all } \, (x_j)_{j \in B} \in H \right\} .$$
When $\mathrm{sign} \equiv 0$, this reduces to the standard dual subgroup  $H^\perp$.

The signature $\mathrm{sign}$ induces a map
$$ \mathrm{sign} : (\widehat{G})^B \to (\widehat{G})^B, \quad (\chi_j)_{j \in B} \mapsto \bigl( (-1)^{\mathrm{sign}(j)} \chi_j \bigr)_{j \in B}.$$

Here, we use additive notation for $\widehat{G}$.

Observe that $(\chi_j)_{j \in B} \in H^\perp$ if and only if $\mathrm{sign}(\chi_j)_{j \in B} \in H^{\perp_{\mathrm{sign}}}$.

An example to keep in mind is the following:  $B=\{0,1\}^A$, $\mathrm{sign}(\omega)= |\omega| \mod 2$. In this case, the Orthogonality lemma \ref{lem:orth} corresponds to the fact that $H^{\perp_{\mathrm{sign}}} = \mathcal{P}^{A,d - l -1}(\widehat{G})$ for $H= \mathcal{P}^{A,l}(G)$.

Recall from \eqref{gray} that the spaces $\mathcal{P}^{A,l}(G)$ were obtained as solution sets to systems of Gray code equations associated to $(l+1)$-dimensional faces of the cube $\{0,1\}^A$. Each such system involves integer coefficients and $2^d$ unknowns in $G$.  Hence, it can be identified with some vector $v \in \mathbb{Z}^{\{0,1\}^A}$ and an associated linear map 
$$ \delta_v : G^{\{0,1\}^A} \to G, \quad \delta_v((x_\omega)) \coloneqq \sum_{\omega \in \{0,1\}^A} v_\omega x_\omega.$$
Moreover, we the unique completion property holds: an element of $\mathcal{P}^{A,l}(G)$ is completely determined by any of its $l$-corners. This corresponds to the fact that the system's matrix can be reduced to echelon form with unit pivots, i.e. its  Smith normal form consisting solely of units.

In general, to keep things simple and avoid technicalities related to  torsions, consider now a lattice  $\Lambda \subseteq \mathbb{Z}^B$ whose basis can be reduced to echelon form with unit pivots. 
Define the corresponding subgroup 
$$\mathcal{P}^{\Lambda}(G) \coloneqq  \bigcap_{v \in \Lambda} \ker \delta_v \subseteq G^B.$$
On $\mathbb{Z}^B$ we have a signed inner product induced by  $\mathrm{sign}$:
$$ \langle v, w \rangle_{\mathrm{sign}} \coloneqq \sum_{j \in B} (-1)^{\mathrm{sign}(j)} v_j w_j.$$
For the lattice $\Lambda$, define its orthogonal complement with respect to this signed inner product as 
$$ \Lambda^{\perp_{\mathrm{sign}}} \coloneqq \left\{ w \in \mathbb{Z}^B \, \middle| \,      \langle v, w \rangle_{\mathrm{sign}} = 0, \text{ for all } v \in \Lambda \right\}.$$

The following proposition shows that the lattice $\Lambda^{\perp_{\mathrm{sign}}}$ generates in $\widehat{G}$ the dual group to $\mathcal{P}^{\Lambda}(G)$ with respect to the signature $\mathrm{sign}$. 

\begin{proposition} Let $G$ be a finite abelian group, $\Lambda \subseteq \mathbb{Z}^B$  a lattice with trivial Smith normal form, and $\mathrm{sign}$ a signature map.
Then the dual group of $\mathcal{P}^{\Lambda}(G)$ with respect to the signature $\mathrm{sign}$ is $\mathcal{P}^\Lambda(G)^{\perp_{\mathrm{sign}}} =  \mathcal{P}^{\Lambda^{\perp_{\mathrm{sign}}}}(\widehat{G})$.
\end{proposition}
\begin{proof} Let $v^{(1)}, \dots, v^{(l)}$ be a basis of $\Lambda$. Fix $\chi \in \widehat{G}$ and a basis vector $v^{(n)}= \left(v^{(n)}_j\right)_{j \in B}$. Consider the $k$-tuple $\mathrm{sign}(\chi v^{(n)}) = \left((-1)^{\mathrm{sign}(j)}v^{(n)}_j \chi\right)_{j \in B}$ in ${\widehat{G}}^B$. Observe that this $k$-tuple lies in $\mathcal{P}^{\Lambda^{\perp_{\mathrm{sign}}}}(\widehat{G})$, since for every $w \in \Lambda^{\perp_{\mathrm{sign}}}$ we have 
$$\delta_w(\mathrm{sign}(\chi v^{(n)}))= \sum_{j \in B}(-1)^{\mathrm{sign}(j)}w_j v^{(n)}_j \chi = 0 \chi = 1.$$
Since the vectors $v^{(1)}, \dots, v^{(l)}$ form a matrix that  can be reduced to echelon form with unit pivots, the subgroup generated by   $\mathrm{sign}(\chi v^{(n)})$'s in $\mathcal{P}^{\Lambda^{\perp_{\mathrm{sign}}}}(\widehat{G})$ isomorphic to ${\widehat{G}}^l$. 

The triviality of the Smith normal form of $\Lambda$ guarantees that $\mathcal{P}^{\Lambda}(G)$ is isomorphic to $G^{k-l}$,  and thus $\mathcal{P}^{\Lambda}(G)^{\perp_{\mathrm{sign}}}$ is isomorphic to ${\widehat{G}}^l$.

Therefore, it remains to verify that   $\mathrm{sign}(\chi v^{(n)}) \in \mathcal{P}^{\Lambda}(G)^{\perp_{\mathrm{sign}}}$. 

\noindent
Let $(x_j)_{j\in B} \in \mathcal{P}^{\Lambda}(G)$. In particular, we have $\delta_{v^{(n)}}(x_j) = 0$, therefore
$$\prod_{j \in B}  \mathcal{C}^{\mathrm{sign}(j)} \left((-1)^{\mathrm{sign}(j)}v^{(n)}_j \chi\right)(x_j) = \prod_{j \in B}  v^{(n)}_j \chi( x_j) =    \chi\left(\sum_{j \in B}v^{(n)}_j x_j\right)  = 1,$$
and the claim follows.
\end{proof}


\if{}

\begin{example} For $k=3$ consider a lattice $\Lambda=\text{span}[(1,-2,1)]$ and signature $\mathrm{sign} = (0,0,0)$. Then $\Lambda^\perp=\text{span}[(2,1,0), (-1,0,1)]$, 
$$\mathcal{P}^{\Lambda}(G) = \left\{ (x, y, z) \in G^3 \, \middle| \, x+z=2y \right\} = \left\{ (2y - z, y, z) \, \middle| \, y,z \in G \right\} =$$ $$= \left\{ (x, x+d, x+2d) \, \middle| \, x,d \in G \right\}, $$ 
and 
$$\mathcal{P}^{\Lambda^\perp}(\widehat{G}) =  \left\{ (r, -2r, r) \, \middle| \, r \in \widehat{G} \right\}.$$ 
The dual nature of $\mathcal{P}^{\Lambda}(G)$ and $\mathcal{P}^{\Lambda^\perp}(\widehat{G})$ then underlies Roth-type Fourier identity
$$  \mathbb{E}_{\substack{x+z=2y \\x,y,z\in G}} f_1(x)f_2(y)f_3(z) = \sum_{r \in \widehat{G}} \widehat{f}_1 (r) \widehat{f}_2 (-2r) \widehat{f}_3 (r).$$ 

For $k=4$ the choice of a lattice $\Lambda=\text{span}[(1,-1,-1,1)]$ and signature $\mathrm{sign} = (0,1,1,0)$ leads to the Gowers inner product of order 2, as:  
$$\Lambda^{\perp_{\mathrm{sign}}} = \text{span}[(1,0,0,-1), (0,1,0,-1), (0,0,1,-1)],$$
$$\mathcal{P}^{\Lambda}(G) = \left\{ (x, y, z, w) \in G^4 \, \middle| \, x-y-z+w =0 \right\} =$$ $$= \left\{ (x, x+a, x+b, x+a+b) \, \middle| \, x,a,b \in G \right\}, $$ 
$$\mathcal{P}^{\Lambda^{\perp_{\mathrm{sign}}}}(\widehat{G}) =  \left\{ (r, r, r, r) \, \middle| \, r \in \widehat{G} \right\}.$$ In this case, the duality gives the identity:
$$  \mathbb{E}_{x,a,b\in G} f_1(x)\overline{f_2(x+a) f_3(x+b)}f_4(x+a+b) = \sum_{r \in \widehat{G}} \widehat{f}_1 (r) \overline{\widehat{f}_2 (r) \widehat{f}_3 (r) }\widehat{f}_4(r).$$
\end{example}

\fi{}

After this preparation, we are in position to state and prove the higher order Poisson summation formula.
 
\begin{theorem}[Higher order Poisson summation formula]
\label{thm:Poisson}
Let $G$ be a finite abelian group, $B$ a finite index set with $|B|=k$, and $\mathrm{sign}: B \to \{0,1\}$ a signature map. For any $k$-tuple of complex-valued functions $(f_j)_{j \in B}$ on $G$, with Fourier transforms $(\widehat{f}_j)_{j \in B}$, and any subgroup $H \leq G^B$, we have:
\begin{equation}
\label{Poisson}
\mathbb{E}_{(x_j) \in H} \prod_{j \in B} \mathcal{C}^{\mathrm{sign}(j)} f_j(x_j) = \sum_{(\chi_j) \in H^{\perp_{\mathrm{sign}}}} \prod_{j \in B} \mathcal{C}^{\mathrm{sign}(j)} \widehat{f}_j(\chi_j).
\end{equation}
\end{theorem}
\begin{proof}
Consider the function
$$
F: G^B \to \mathbb{C}, \quad F\left((x_j)_{j \in B}\right) := \prod_{j \in B} \mathcal{C}^{\mathrm{sign}(j)} f_j(x_j).
$$
By the standard Poisson summation formula for the subgroup $H \leq G^B$, we know:
$$
\mathbb{E}_{x \in H} F(x) = \sum_{\chi \in H^\perp} \widehat{F}(\chi).
$$

Since \(F\) factors as a product of functions and characters on \(G^B\) factor as products of characters on \(G\), its Fourier transform factorizes as:
\[
\widehat{F}\left((\chi_j)_{j \in B}\right) = \prod_{j \in B} \mathcal{C}^{\mathrm{sign}(j)} \widehat{f}_j\left((-1)^{\mathrm{sign}(j)} \chi_j\right).
\]

On the other hand, by the definition of the dual subgroup with respect to  signature, \((\chi_j) \in H^\perp\) if and only if \(\big((-1)^{\mathrm{sign}(j)} \chi_j\big) \in H^{\perp_{\mathrm{sign}}}\)
and the claim of the theorem follows.
\end{proof}

\begin{corollary}
\label{cor:Poisson} Let $(f_j)_{j \in B}$ be a $k$-tuple of complex-valued functions on a finite abelian group $G$, $\Lambda \subseteq \mathbb{Z}^B$  a lattice with trivial Smith normal form, and $\mathrm{sign}$ a signature.
 Then
\begin{equation}
\label{PoissonLattice}
     \mathbb{E}_{(x_j) \in \mathcal{P}^\Lambda(G)}  \prod_{j \in B}  \mathcal{C}^{\mathrm{sign}(j)} f_j(x_j) = \sum_{(\chi_j) \in \mathcal{P}^{\Lambda^{\perp_{\mathrm{sign}}}}(\widehat{G})} \, \prod_{j \in B}  \mathcal{C}^{\mathrm{sign}(j)} \widehat{f}_j(\chi_j). 
\end{equation}
\qed 
\end{corollary}

\begin{remark}
One might wonder which other lattices in $\mathbb{Z}^B$ lead to interesting product averages, particularly if one is interested in norms on $\mathbb{C}^G$ and arithmetic progressions. Numerical experiments suggest that a natural candidate, the inner product given by the degenerate cubic formula 
$$ \mathbb{E}_{x,a\in G} f_1(x)\overline{f_2(x+a) f_3(x+a) f_4(x+a)}f_5(x+2a)f_6(x+2a)f_7(x+2a)\overline{f_8(x+3a)} $$
is not positive definite and, thus, does not define a norm. 
\end{remark}

For the sake of completeness, we include an affine version of higher order Poisson formula, as well as its proof. It appeared in a less general form in \cite{Rak21}.

\begin{proposition}[Higher order Poisson summation formula -- an affine version]
\label{prop:PoissonAf}

Let $(f_j)_{j \in B}$ be a $k$-tuple of complex-valued functions on a finite abelian group $G$, $(\widehat{f}_j)_{j \in B}$ its Fourier transform, $\mathrm{sign}$ a signature, $H$ a subgroup of $G^B$ and $t \in G^B$ a translation. Then
\begin{equation}
\label{PoissonAf}
     \mathbb{E}_{(x_j) \in H + t }  \prod_{j \in B}  \mathcal{C}^{\mathrm{sign}(j)} f_j(x_j) = \sum_{(\chi_j) \in H^{\perp_{\mathrm{sign}}}} \, \prod_{j \in B}  \mathcal{C}^{\mathrm{sign}(j)} \left( \widehat{f}_j(\chi_j)\right) \chi_j((-1)^{\mathrm{sign}(j)} t_j) . 
\end{equation}
\end{proposition} 
\begin{proof} Using a change of variables $x= (x_j) = (y_j) + (t_j) = y + t$ for $x\in H+t$ and $y \in H$ the left-hand side becomes:
\[ \mathbb{E}_{(x_j) \in H + t }  \prod_{j \in B}  \mathcal{C}^{\mathrm{sign}(j)} f_j(x_j) =
 \mathbb{E}_{(y_j) \in H }  \prod_{j \in B}  \mathcal{C}^{\mathrm{sign}(j)} f_j(y_j + t_j).\]
Under Fourier transform, the translation in $f_j$ corresponds to a phase shift in $\widehat{f}_j$. Accounting for complex conjugation, we get directly from Theorem \eqref{thm:Poisson} :
\[\mathbb{E}_{(y_j) \in H }  \prod_{j \in B}  \mathcal{C}^{\mathrm{sign}(j)} f_j(y_j + t_j) = \sum_{(\chi_j) \in H^{\perp_{\mathrm{sign}}}} \, \prod_{j \in B}  \mathcal{C}^{\mathrm{sign}(j)} \left( \widehat{f}_j(\chi_j) \chi_j(t_j) \right) = \]
\[= \sum_{(\chi_j) \in H^{\perp_{\mathrm{sign}}}} \, \prod_{j \in B}  \mathcal{C}^{\mathrm{sign}(j)} \left( \widehat{f}_j(\chi_j)\right) \chi_j((-1)^{\mathrm{sign}(j)} t_j).\]
\end{proof}


 {\it Corner convolutions}, or more generally {\it $U^d$-convolutions}, were used implicitly by Host and Kra in \cite{HoKr04}, \cite{HoKr05},  and later studied explicitly by Szegedy in \cite{Szeg10d}, \cite{Szeg12}, and by Candela and Szegedy in \cite{CanSze18}. In these works,  the $2^d$-ary inner product $\langle F \rangle_{d,1}$ is expressed as an ordinary inner product of between the function $f_0$, corresponding to corner $0\dots 0$,  and a composite  function $\dot{K}(F')$, where $F'$ denotes a cube of functions $F$ with $f_0$ removed:
$$ \left\langle f_0, F' \right\rangle_{d,1}  = \left\langle f_0, \overline{\dot{K}(F')} \right\rangle_{1,0}.$$
The corner convolution $\dot{K}(F')$ is given by averaging over $h_1, \dots , h_d \in G$:
$$ \dot{K}(F')(x) \coloneqq \mathbb{E}_{h_1, \dots , h_d} \prod_{\omega \neq 0} \mathcal{C^{|\omega|}}f_\omega(x + \sum_i \omega_i h_i) .$$
Recall that in \eqref{DiffP(A,l)} we described a splitting $\mathcal{P}^{d,1}(G) = \mathcal{P}^{d,0}(G) \times \mathcal{D_0}\mathcal{P}^{d,1}(G)$, so  a $d$-tuple of parameters $h_1, \dots, h_d$ determines a unique point $r = (y_\omega) \in \mathcal{D}_0\mathcal{P}^{d,1}(G)$ with $y_\omega = \sum_i \omega_i h_i$. This is a $d$-dimensional parallelepiped in $G$ with $y_{0\dots 0} = 0$, generated by sides $h_1, \dots, h_d$. To obtain all parallelepipeds in $\mathcal{P}^{d,1}(G)$, we need to shift $r$ by a constant cube $p=(x_\omega)\in \mathcal{P}^{d,0}(G)$, i.e.,
$x_\omega = x_{0\dots0}$ for all $\omega \in \{0,1\}^A$.

If we modify our notation of points and multiplicative derivatives slightly with an added dot to reflect the removal of the corner $0 \dots 0$ from an $A$-cube, we get
$$ \dot{K}(F')(x_{0\dots0}) = \mathbb{E}_{r \in  \mathcal{D}_0\mathcal{P}^{d,1}(G)} \overline{\Delta}^A_{\dot{p}+\dot{r},1} (F') = \mathbb{E}_{r \in  \mathcal{D}_0\mathcal{P}^{d,1}(G)} \prod_{\substack{\omega \in \{0,1\}^A\\ \omega \neq  0\dots0 }} \mathcal{C}^{|\omega|} f_\omega (x_{0\dots0} + y_\omega).$$

The same can be done for the cube of Fourier transforms $\hat{F}$ and the $(d,d-2)$-inner product $\langle \hat{F} \rangle_{d,d-2}$. Here we use the splitting  from \eqref{Diff2P(A,l)}: $\mathcal{P}^{d,d-2}(\hat{G}) = \mathcal{P}^{d,0}(\hat{G}) \times \mathcal{D}_0\mathcal{P}^{d,d-2}(\hat{G})$, i.e., each degree $d-2$ cube of characters can be shifted by a constant cube of characters $P= (\chi_\omega)$ to $R=(\rho_\omega)$ of degree $d-2$ with $\rho_{0\dots0}=1 \in \hat{G}$. Such $R$ is then uniquely determined by values $\rho_\omega$ for $1\le |\omega|\le d-2$. Defining a convolution as an average of multiplicative derivatives over such $R$'s, we get
$$\dot{K}(\hat{F}) (\chi_{0\dots0}) \coloneqq  \sum_{R \in  \mathcal{D}_0\mathcal{P}^{d,d-2}(\hat{G})} \overline{\Delta}^A_{\dot{P}+\dot{R},d-2} (\hat{F}') = $$
$$ = \sum_{R \in  \mathcal{D}_0\mathcal{P}^{d,d-2}(\hat{G})} \prod_{\substack{\omega \in \{0,1\}^A \\ \omega \neq 0\dots0} } \mathcal{C}^{|\omega|} \hat{f}_\omega (\chi_{0\dots0} + \rho_\omega).$$

Using Parseval's identity and the duality of $(d,l)$-inner products, we obtain
$$ \langle \hat{f_0}, \widehat{\overline{\dot{K}(F')}} \rangle_{1,0} = \langle f_0, \overline{\dot{K}(F')} \rangle_{1,0} = \langle f_0, F' \rangle_{d,1}  = \langle \hat{f_0}, \hat{F}' \rangle_{d,d-2}  = \langle \hat{f_0}, \overline{\dot{K}(\hat{F}')} \rangle_{1,0},$$
giving a Fourier identity between corner convolutions $\dot{K}(F')$ and $\dot{K}(\hat{F}')$:
\begin{equation}
    \label{Corner}
    \widehat{\dot{K}(F')}   (\chi_{0\dots0}) = \dot{K}(\hat{F}) (\chi_{0\dots0}) =\sum_{R \in  \mathcal{D}_0\mathcal{P}^{d,d-2}(\hat{G})} \prod_{\substack{\omega \in \{0,1\}^A\\ \omega \neq  \{0\dots0\} }} \mathcal{C}^{|\omega|} \hat{f}_\omega (\chi_{0\dots0} + \rho_\omega).
\end{equation}

We would obtain the same result using the affine Poisson summation formula \eqref{PoissonAf}, checking along the way that the signed dual to $\dot{\mathcal{D}}_{0}\mathcal{P}^{d,1}(G)$ inside $\dot{G}^{\{0,1\}^A}$ is precisely $\dot{\mathcal{D}}_{0}\mathcal{P}^{d,d-2}(\hat{G})$ in $\dot{\hat{G}}^{\{0,1\}^A }$.

The same argument works for $l\ge 2$, and we get that the Fourier transform of $U^{d,l}$-convolution of $F$ is the $U^{d,d-l-1}$-convolution of its dual $\hat{F}$.

\bibliographystyle{amsplain}

\end{document}